\documentclass[a4paper,reqno,10pt]{amsart}
\usepackage{mathtools,amsthm,amsmath}
\usepackage[pdfencoding=auto]{hyperref}
\usepackage{xcolor}

\AddToHook{bfseries}{\boldmath}

\theoremstyle{plain}
\newtheorem*{theorem*}{Theorem}
\newtheorem{theorem}{Theorem}[section]

\newtheorem{proposition}[theorem]{Proposition}
\newtheorem{lemma}[theorem]{Lemma}

\theoremstyle{definition}

\theoremstyle{remark}
\newtheorem{remark}[theorem]{Remark}

\title[The \texorpdfstring{$\nu$}{nu}-invariant of two-step nilmanifolds with closed \texorpdfstring{$\mathrm G_2$}{G2}-structure]{On the \texorpdfstring{$\nu$}{nu}-invariant of two-step nilmanifolds with closed \texorpdfstring{$\mathrm G_2$}{G2}-structure}

\author{Anna Fino}
\address{Universit\`a degli Studi di Torino \\
Dipartimento di Matematica ``G.\ Peano'' \\
via Carlo Alberto 10 \\
10123 Torino, Italy}
\email{annamaria.fino@unito.it}

\author{Gueo Grantcharov}
\address{Department of Mathematics and Statistical Sciences \\
Florida International University \\
11200 SW 8th Street \\
FL33199, Miami \\
United States
\&
Institute of Mathematics and
Informatics, Bulgarian Academy of Sciences\\
8 Acad.\ Georgi Bonchev Str.\ 1113 Sofia, Bulgaria}
\email{grantchg@fiu.edu}

\author{Giovanni Russo}
\address{Mathematics Area, SISSA - Scuola Internazionale Superiore di Studi Avanzati,
via Bonomea 265, 34136 Trieste (TS), Italy}
\email{girusso@sissa.it}

\newcommand{\hatotimes}{\,\hat \otimes \,}

\begin{document}

\begin{abstract}
For every non-vanishing spinor field on a Riemannian spin seven-manifold, Crowley, Goette, and Nordstr\"om defined the so-called $\nu$-invariant. 
This is an integer modulo $48$ that detects connected components of the moduli space of $\mathrm G_2$-structures on any seven-dimensional oriented spin manifold.
The $\nu$-invariant can be defined in terms of Mathai--Quillen currents, harmonic spinors, and $\eta$-invariants of spin Dirac and odd-signature operator. 
We compute these data for certain families of left-invariant closed $\mathrm G_2$-structures on compact two-step nilmanifolds with their natural spin structure.
Specifically, we establish the existence of non-invariant harmonic spinors and determine the parity of the dimension of the space of harmonic spinors. 
We deduce the vanishing of $\nu$ on invariant harmonic spinors.
\end{abstract}

\maketitle

\section*{Introduction}
\label{sec:introduction}
A \emph{$\mathrm G_2$-structure} on a seven-manifold $M$ is a reduction of  the structure group of the frame bundle of $M$ to the compact, exceptional Lie group $\mathrm G_2$. 
This can be defined as a subgroup of the special orthogonal group $\mathrm{SO}(7)$ fixing a certain three-form $\varphi_0$. 
Equivalently, a $\mathrm G_2$-structure on $M$ is a choice of a positive three-form $\varphi$ on $M$ linearly equivalent to $\varphi_0$ pointwise. 
Such a three-form $\varphi$ induces a Riemannian metric $g_{\varphi}$ and an orientation  on $M$, whence a Hodge star operator $\star_{\varphi}$.
Fern\'andez and Gray \cite{fernandez-gray} obtained a classification of $\mathrm G_2$-structures into $16$ different classes in terms of the irreducible $\mathrm G_2$-components of $\nabla \varphi$,  where $\nabla$ is  the Levi-Civita connection of the metric $g_{\varphi}$. The result can be described  in terms of the exterior derivatives  $d \varphi$ and $d(\star_{\varphi} \varphi)$, cf.\ Bryant \cite{bryant}.  If   $\nabla \varphi =0$, the restricted holonomy group of $g_{\varphi}$ is contained in $\mathrm G_2$ and the three-form  $\varphi$ is closed and coclosed (equivalently \emph{parallel}).
When $\varphi$ is only  closed, the  $\mathrm G_2$-structure  is  called \emph{closed} (or \emph{calibrated}), and in this case $d (\star_{\varphi} \varphi) = \tau_2 \wedge \varphi$, for some non-zero two-form $\tau_2$. 
 
Many examples of closed non-parallel $\mathrm G_2$-structures were obtained on compact quotients of simply connected Lie groups by cocompact lattices. 
The first one was constructed by Fern\'andez \cite{fernandez} on a \emph{two-step nilmanifold}, i.e.\ a compact quotient of a simply connected two-step nilpotent Lie group by a lattice. Nilpotent Lie algebras admitting invariant closed $\mathrm G_2$-structures have been classified by Conti--Fern\'andez \cite{conti-fernandez}. Classification results for solvable Lie algebras with non-trivial center and in the non-solvable case can be found in \cite{fino-raffero,fino-raffero-salvatore}. 

It is well-known that a seven-manifold $M$ equipped with a $\mathrm G_2$-structure is necessarily spin with a natural spin structure, so we fix this spin structure over $M$.
Recall that the standard $\mathrm{Spin}(7)$-representation is eight-dimensional and of real type.
The associated rank $8$ vector bundle $\pi \colon SM \to M$ is then real, and comes with a Riemannian metric $g^S$ induced by $g$.
The group $\mathrm{G}_2$ is simply connected, so the inclusion $\mathrm{G}_2 \hookrightarrow \mathrm{SO}(7)$ lifts to an inclusion $\mathrm{G}_2 \hookrightarrow \mathrm{Spin}(7)$.
This gives an action of $\mathrm{G}_2$ on the standard (real) $\mathrm{Spin}(7)$-representation, which then splits as the direct sum of the standard seven-dimensional $\mathrm{G}_2$-representation and a trivial summand.
A unit-length spinor $\phi$ (unique up to a sign) taking values in the latter summand is said to be compatible with the $\mathrm{G}_2$-structure $\varphi$.
Conversely, any unit spinor on $M$ defines a compatible $\mathrm{G}_2$-structure in a unique way. 
It turns out that $\varphi$ is closed (non-parallel) if and only if any attached spinor field $\phi$ is \emph{harmonic} with respect to the spin Dirac operator $D_M$, namely $D_M\phi=0$, cf.\ \cite[Section 4, Theorem 4.6]{agricola}.
 
For any operator $A$,  denote by $\eta(A)$ the Atiyah--Patodi--Singer \emph{$\eta$-invariant} \cite{atiyah-patodi-singer}, i.e.\ the sum 
\begin{equation}
\label{eq:def-eta-invariant}
\eta(A) = \sum_{\lambda \neq 0} \mathrm{sgn}(\lambda),
\end{equation}
where $\lambda$ is an eigenvalue of $A$, and by $h(A)$ the dimension of the kernel of $A$. Following Crowley--Goette--Nordstr\"om \cite[Theorem 1.2]{crowley-goette-nordstrom}, for any non-vanishing spinor field $\phi$ on a Riemannian manifold $(M,g)$ one can define the so-called \emph{$\nu$-invariant} in the following way:
\begin{equation}
\label{def:nu-invariant}
\nu(\phi) \coloneqq 2\int_M \phi^*\psi(\nabla^S,g^S) - 24(\eta+h)(D_M)+3\eta(B_M) \in \mathbb Z/48.
\end{equation}
Here $\psi(\nabla^S,g^S)$ is the \emph{Mathai--Quillen current} \cite{mathai-quillen}, (cf.\ also \cite{bismut-zhang,wu}), and $B_M$ is the odd signature operator acting on differential forms of even-degree on $M$ as $B_M\omega \coloneqq (-1)^{p+1}(\star_{\varphi} d-d\star_{\varphi})\omega$, $\omega \in \Omega^{2p}(M)$.
The above description \eqref{def:nu-invariant} of the $\nu$-invariant relies on the Atiyah--Patodi--Singer Index Theorem (see \cite{crowley-goette-nordstrom}), but $\nu$ was originally defined in terms of a spin cobordism: this is to say that a manifold $M$ with a $\mathrm G_2$-structure appears as the boundary of a spin eight-manifold $W$, and $\nu$ can be calculated in terms of topological and spinorial data on $W$, cf.\ \cite[Definition 1.1]{crowley-goette-nordstrom} and \cite[Section 3.2]{crowley-nordstrom}. However, the value of $\nu$ is independent of the choice of $W$ modulo $48$ \cite[Corollary 3.2]{crowley-nordstrom}.
In general, $\nu(\phi)$ does not change under continuous variations of $\phi$ (while keeping the metric fixed) as long as $\phi$ remains never-vanishing.
The role of $\nu$ is to detect connected components of the moduli space of $\mathrm G_2$-structures.
An explicit computation of its values for certain holonomy $\mathrm G_2$ metrics can be found in \cite{crowley-goette-nordstrom, scaduto}.

The purpose of the paper is two-fold.
Our main goal is to compute the $\nu$-invariant for certain families of left-invariant closed $\mathrm G_2$-structures on compact two-step nilmanifolds, thus investigating the non-integrable case. 
Motivated by this problem, we find relevant information on the dimension of the kernel of the Dirac operator.
It turns out that harmonic spinors always form a space of even dimension on our compact nilmanifolds with trivial spin structure.
We also show that non-invariant harmonic spinors exist.

By the classification in \cite[Theorem 4]{conti-fernandez}, it turns out that, up to isomorphism, there exist only two seven-dimensional two-step nilpotent (non-Abelian) Lie algebras admitting closed  $\mathrm G_2$-structures, which have structure equations
\begin{equation}
\label{eq:lie-algebras-streq}
\mathfrak h_1 \coloneqq  (0,0,0,0,0,12,13), \qquad \mathfrak h_2 \coloneqq (0,0,0,12,13,23,0).
\end{equation}
The notation $(0,0,0,0,0,12,13)$ means there is a basis $(E^1, \ldots, E^7)$ of the dual Lie algebra satisfying $d E^j =0$, $j =1,\ldots,5$, $d E^6 = E^{12}$ and $d E^7 = E^{13},$ where $E^{ij}$ is a shorthand for $E^i\wedge E^j$, and similarly for $\mathfrak h_2$.
We will use the same convention on the indices for wedge products of forms of different degree.
For nilmanifolds associated to the Lie algebras in \eqref{eq:lie-algebras-streq}, the moduli space of invariant closed $\mathrm G_2$-structures has been studied by Bazzoni--Gil-Garc\'ia \cite{bazzoni-gilgarcia}.
It is shown that the dimension of this moduli space in the case of $\mathfrak h_1$ is less than or equal to $1$, and for $\mathfrak h_2$ it is exactly $3$.

On any associated two-step nilmanifold $M$, one can then look for the invariant spinor field (unique up to a sign) compatible with a chosen invariant closed $\mathrm G_2$-structure. 
Since such a spinor is harmonic, we have $h(D_M)>0$.
We find the parity of $h(D_M)$ by using the representation theory of  $\mathfrak h_1$ and $\mathfrak h_2$, in the spirit of Ammann--B\"ar \cite{ammann-bar} and Gornet--Richardson \cite{gornet-richardson}.
For the connected, simply connected Lie group with Lie algebra isomorphic to $\mathfrak h_1$, the results will depend only on the invariant metric rather than the chosen $\mathrm G_2$-structure and lattice. 
For the group corresponding to $\mathfrak h_2$, we do the computations for a specific choice of a lattice, as a classification of lattices in this case does not seem to be known.
A priori, the results depend on the spin structure chosen as well, but we fix the trivial one in both cases and work with it.
In the case of $\mathfrak h_1$, we use a general result on the equivalence of the different invariant scalar products under the diffeomorphism group \cite[Theorem 4.9]{reggiani-vittone}, together with a calculation of $D_M$ for a one-parameter family of left-invariant metrics.
We deduce that $h(D_M)$ is even, whence $24h(D_M)=0 \pmod{48}$.
We provide evidence that the latter result is true also for a certain family of scalar products on $\mathfrak h_2$, related to the family considered by Nicolini \cite{nicolini}.
Nilmanifolds admitting invariant harmonic spinors were also studied in \cite{bazzoni-lucia-munoz}.

The specific structure  of $\mathfrak h_1$ and $\mathfrak h_2$ allows us to define an orientation-reversing isometry on the corresponding nilmanifolds, which yields the vanishing of the $\eta$-invariants $\eta(D_M)$ and $\eta(B_M)$ at once, cf.\ \cite[Proposition 1.5]{crowley-goette-nordstrom}. 
Lastly, we set up the (super) linear algebra we need to compute the Mathai--Quillen currents for any harmonic spinor compatible with a left-invariant closed $\mathrm G_2$-structure, and show that it always vanishes. 
Our main result is the following (cf.\ \eqref{eq:lie-algebras-streq} for structure equations and notations).
\begin{theorem*}
Let $\mathfrak g$ be a two-step nilpotent seven-dimensional Lie algebra admitting a closed  $\mathrm G_2$-structure. 
Let $G$ be the corresponding connected, simply connected Lie group, and let $M$ denote a quotient of $G$ by some cocompact lattice $\Gamma$.
The following hold for the trivial spin structure on $M$ and up to automorphism of $\mathfrak g$.
\begin{enumerate}
\item If $\mathfrak g \cong \mathfrak h_1$, then the $\nu$-invariant $\nu(\phi)$ vanishes for any invariant Riemannian metric on any nilmanifold $M=\Gamma\backslash G$ induced by the one-parameter family of invariant closed $\mathrm G_2$-structures
\begin{align*}
\varphi_{a} & = E^{123}+E^{145}+a^2E^{167}+aE^{246} -aE^{257}-aE^{356}-aE^{347},
\end{align*}
and for any invariant harmonic spinor $\phi$.
\item If $\mathfrak g \cong \mathfrak h_2$, then the $\nu$-invariant $\nu(\phi)$ vanishes for any invariant Riemannian metric on one nilmanifold $M=\Gamma\backslash G$ induced by the two-parameter family of invariant closed $\mathrm G_2$-structures 
\begin{align*}
\varphi_{b_1,b_2} & = E^{123}+(b_1+b_2)b_1E^{145}+b_2E^{167}+(b_1+b_2)b_2E^{246} \\
& \qquad -b_1E^{257}-(b_1+b_2)E^{347}-b_1b_2E^{356}, 
\end{align*}
and for any invariant harmonic spinor $\phi$.
\end{enumerate}
\end{theorem*} 
The specific form of the lattices will be given in detail later.
We remark that the families considered above do not include all possible invariant closed $\mathrm G_2$-structures on $\mathfrak g$.
In case (2), the two-parameter family we consider is the one treated by Nicolini \cite{nicolini}, which is a continuous family of Laplacian solitons that are pairwise non-homothetic.
In order to have a better understanding of the $\nu$-invariant, it would be useful to compute it for higher step nilmanifolds. 
For the moment, this question remains open.

The structure of the paper is as follows. In Section \ref{preliminaries} we set up all general preliminaries needed, i.e.\ spinor calculus and spin geometry in dimension seven, the representation theory of nilpotent Lie groups (Kirillov theory) we are going to use, and the relevant linear algebra for the Mathai--Quillen currents. In Section \ref{sec:nilmanifold1} we compute harmonic spinors, $\eta$-invariants, and Mathai--Quillen currents on nilmanifolds obtained from $\mathfrak h_1$, and discuss all relevant details in order to get to case (1) of the above Theorem. 
Section \ref{sec:nilmanifold2} proceeds similarly, but we only highlight the main differences to the arguments given in Section \ref{sec:nilmanifold1}. 
The content of Section \ref{sec:nilmanifold1} and \ref{sec:nilmanifold2} is a proof of the above Theorem. 

\textbf{Acknowledgments}. The authors would like to thank Ciprian Gal, Giovanni Bazzoni, Johannes Nordstr\"om, and Simon Salamon for useful discussions. We also thank the anonymous reviewer for interesting comments improving the exposition. First and third named authors are partially supported by INdAM--GNSAGA. 
Anna Fino is also partially supported by Project PRIN 2022 \lq\lq Geometry and Holomorphic Dynamics\rq\rq, by a grant from the Simons Foundation (\#944448), and by the NSF Grant (\#DMS-1928930), during a visit at the Simon Laufer Mathematical Sciences Institute (formerly MSRI) in Berkeley, California, in Fall 2024. Gueo Grantcharov is partially supported by a grant from the Simons Foundation (\#853269).
Giovanni Russo is partially supported by the PRIN 2022 Project (2022K53E57) - PE1 - \lq\lq Optimal transport: new challenges across analysis and geometry\rq\rq\ funded by the Italian Ministry of University and Research.

\section{Preliminaries}
\label{preliminaries}

We review the spin geometry in dimension seven we are going to use throughout.
To set things up, we rely on the conventions in \cite{bfgk, friedrich}. 
Let $e_1,\dots,e_7$ be the canonical basis of $\mathbb R^7$ equipped with its standard scalar product.
Let $\mathrm{Cl}(7)$ be the real Clifford algebra generated by $e_1,\dots,e_7$ via the relation $e_ie_j+e_je_i = -2\delta_{ij}1$, where $\delta_{ij}$ is the Kronecker delta.
Recall that the complexified Clifford algebra $\mathrm{Cl}^{\mathbb C}(7)$ is isomorphic to the algebra $\mathrm{End}(\mathbb{C}^8) \oplus \mathrm{End}(\mathbb{C}^8)$.
Let us write $\Delta$ for the $8$-dimensional representation of $\mathrm{Cl}^{\mathbb C}(7)$ we get from the first summand in $\mathrm{End}(\mathbb C^8) \oplus \mathrm{End}(\mathbb C^8)$.
In this dimension, $\Delta$ has a real structure \cite[Section 1.7]{friedrich}, so we write $[\Delta] \simeq \mathbb R^8$ for this real representation.
Restricting the action of $\mathrm{Cl}^{\mathbb C}(7)$ to $\mathrm{Spin}(7) \subset \mathrm{Cl}(7)$ realises $[\Delta]$ as a $\mathrm{Spin}(7)$-representation.
 
\subsection{Spin geometry}
\label{sec:spin-geometry-dimension-seven}

Let $(M,g)$ be an oriented Riemannian seven-manifold.
Assume $(M,g)$ is spin, i.e.\ that there is a lifting of the principal bundle of orthonormal frames with group $\mathrm{SO}(7)$ to a principal bundle $P$ with structure group $\mathrm{Spin}(7)$.
It is well-known that $M$ is spin when its second Stiefel--Whitney class vanishes, and spin structures are classified by $H^1(M,\mathbb Z_2)$.

Let us now fix a spin structure $P$ over $M$.
The spinor bundle $P \times_{\mathrm{Spin}(7)} \Delta$ is the complex vector bundle over $M$ associated to the chosen spin structure via $\Delta$, and its sections are spinor fields.
Clearly, it inherits a Hermitian product induced by $\langle{}\cdot{},{}\cdot{}\rangle$ fibrewise.
Since $\Delta$ admits a real structure, we may restrict to real spinor fields, i.e.\ sections of $SM \coloneqq P \times_{\mathrm{Spin}(7)} [\Delta]$.
The latter real vector bundle inherits a Riemannian structure which we denote again by $({}\cdot{},{}\cdot{})$. 

The Levi-Civita connection $\nabla$ on $(M,g)$ induces a natural connection $\nabla^S$ on $SM$.
Let $\Gamma(SM)$ be the space of smooth sections of $\pi \colon SM \to M$. 
Any local section of the bundle of orthornormal frames on an open set $U \subset M$ lifts to the spin structure $P$ defining a trivialisation of $SM$, so $SM_{|U}=U \times [\Delta]$.
With respect to this trivialisation, $\nabla^S$ acts on a local section $\phi$ of $SM$ via the formula
\begin{equation}
\label{def:spin-cov-der}
\nabla_{e_i}^S\phi = \partial_{e_i}\phi + \tfrac12 \sum_{j<k} g(\nabla_{e_i}e_j,e_k)e_je_k \phi,
\end{equation}
where $\{e_{\ell}:\ell=1,\dots,7\}$ is a local orthonormal frame, and the product $e_k\phi$ is Clifford multiplication.
Let $R$ and $R^S$ be the Riemannian curvature tensors attached to $\nabla$ and $\nabla^S$ respectively. 
We may view $R$ as an operator $\Lambda^2 TM \to \Lambda^2 TM$, and similarly $R^S$ as an operator $\Lambda^2 TM \to \Lambda^2 SM$. 
The two are related by the formula $2R^S(x,y) \phi = \sum_{j<k} g(R(x,y)e_j,e_k)e_je_k\phi$, $x,y \in \mathfrak{X}(M)$. 
We refer to \cite{bfgk} for details.

The spin Dirac operator $D_M \colon \Gamma(SM) \to \Gamma(SM)$ acts on sections of $SM$ and is given locally by the formula 
\begin{equation}
\label{eq:definition-dirac-operator}
D_M \phi \coloneqq \sum_{k} e_k \nabla_{e_k}^S\phi,
\end{equation}
where we have used the same notations as above for a local orthonormal frame.
We recall that $D_M$ is a first order elliptic operator. 
When $M$ is compact, $D_M$ is formally self-adjoint in the space $L^2(SM)$ of $L^2$-sections of $SM \to M$ equipped with the scalar product 
$(\phi_1,\phi_2)_{L^2} \coloneqq \int_M \langle \phi_1,\phi_2\rangle \, d\mu_M$. 
Further, in this case the spectrum of $D_M$ is discrete and consists of real eigenvalues with finite multiplicity, and $\ker D_M$ is finite-dimensional.
Elements of $\ker D_M$ are called harmonic spinors.

Lastly, let us review the correspondence between $\mathrm{G}_2$-structures and spinors in dimension seven, cf.\ for instance Agricola et al.\ \cite{agricola}. 
Let us consider $\mathbb R^7$ with its standard basis $e_1,\dots,e_7$ as above. Let $e^1,\dots,e^7$ be the dual forms, and define the three-form
\[\varphi_0 \coloneqq e^{123}+e^{145}+e^{167}+e^{246}-e^{257}-e^{347}-e^{356},\]
where $e^{ijk}$ is a shorthand for $e^i \wedge e^j \wedge e^k$. The group $\mathrm{G}_2$ can be defined as the stabiliser of $\varphi_0$ inside $\mathrm{SO}(7)$.
It is well-known that $\mathrm{G}_2$ is compact, connected, simply connected, $14$-dimensional, and simple \cite{bryant}.
Since $\mathrm{G}_2$ is simply connected, the immersion $\mathrm{G}_2 \hookrightarrow \mathrm{SO}(7)$ lifts to the double cover $\mathrm{G}_2 \hookrightarrow \mathrm{Spin}(7)$.
Therefore, $\mathrm{G}_2$ acts on the real spin representation $[\Delta$].

A $\mathrm{G}_2$-structure on an oriented Riemannian seven-manifold $(M,g)$ is a reduction $Q$ of the principal bundle of orthonormal frames to the Lie group $\mathrm{G}_2$.
This is equivalent to saying that $(M,g)$ carries a three-form $\varphi$ pointwise linearly equivalent to $\varphi_0$ and compatible with metric and orientation.
The exact relationship is given by 
\begin{equation}
\label{eq:relationship-g2-metric-vol}
g(x,y)\mathrm{vol}_g = \tfrac16(x\lrcorner\, \varphi) \wedge (y\lrcorner\, \varphi) \wedge \varphi, \qquad x,y \in \mathfrak{X}(M),
\end{equation}
where $\mathrm{vol}_g$ is a Riemannian volume form.
It is well-known that $(M,g,\varphi)$ is automatically spin and has a natural spin structure $Q \times_{\mathrm G_2} \mathrm{Spin}(7)$, so we work with this.
Let $SM$ be the associated spinor bundle with fibre $[\Delta]$ as above.
Since $\mathrm{G}_2$ acts on $[\Delta]$, we have a splitting into irreducible summands $V \oplus \mathbb R$ (here $V$ is the standard seven-dimensional $\mathrm{G}_2$-representation, and $\mathbb R$ is the trivial one).
The latter line is generated by a unit spinor (up to a sign), which then induces a spinor field $\phi$ of unit length on $M$. 
Conversely, suppose $(M,g,\phi)$ is a Riemannian seven-manifold with a globally defined unit (real) spinor $\phi$.
Then a $\mathrm{G}_2$-structure $\varphi_{\phi}$ is induced on $M$ via the following formula: 
\begin{equation}
\label{eq:three-form-spinor}
\varphi_{\phi}(x,y,z) \coloneqq (xyz \phi,\phi), \qquad x,y,z \in \mathfrak{X}(M).
\end{equation}
That $\varphi_{\phi}$ is skew-symmetric follows by the defining property of $\mathrm{Cl}(7)$.

Let us now fix a unit spinor $\phi$. By the properties of $\nabla^S$ (cf.\ \cite{bfgk}) one finds that there is an endomorphism $S \in \mathrm{End}(TM)$ such that 
\begin{equation}
\label{eq:intrinsic-endomorphism}
\nabla_x^S \phi = S(x)\phi.
\end{equation}
The tensor $S$ is known as the \emph{intrinsic endomorphism} of $(M,g,\phi)$. 
We refer to \cite{agricola} and \cite{fernandez-gray} for classification results and related properties of $S$. 
We will only be interested in closed $\mathrm{G}_2$-structures, i.e.\ those $\mathrm{G}_2$-structures whose three-form $\varphi$ satisfies $d\varphi=0$.
A remarkable fact is that a unit spinor $\phi$ associated to a closed $\mathrm{G}_2$-structure is harmonic \cite[Theorem 4.6]{agricola},
in which case $S \in \mathfrak g_2 \subset \mathfrak{so}(7)$.

\subsection{Representation theory of nilpotent Lie groups}
\label{subsec:representation-theory-nilpotent-groups}

The action of the Dirac operator on spinors defined on a nilmanifold $\Gamma \backslash G$ can be worked out via the representation theory of $G$.
The latter is achieved by using Kirillov theory, as $G$ is nilpotent.
The essential point of the theory is the one-to-one correspondence between irreducible representations of $G$ and coadjoint orbits of $G$. 
When dealing with representations, it will be convenient to work over the complex numbers, and then reduce considerations to the reals only at the end.
The main general reference we rely on is Corwin--Greenleaf~\cite{corwin-greenleaf}. 
Useful applications are found in Ammann--B\"ar~\cite{ammann-bar} and Gornet--Richardson~\cite{gornet-richardson}.

Any Lie group $G$ acts on its Lie algebra $\mathfrak g$ via the adjoint action $\mathrm{Ad}_g(x) = gxg^{-1}$, for $g \in G$ and $x \in \mathfrak g$.
The corresponding action on the dual Lie algebra $\mathfrak g^*$ is the coadjoint action $\mathrm{Ad}^*_g(\ell) = \ell(\mathrm{Ad}_{g^{-1}}{}\cdot{})$, for $g \in G$ and $\ell \in \mathfrak g^*$.
Coadjoint orbits are $G$-orbits in $\mathfrak g^*$. 
The differential of $\mathrm{Ad}^*$ is a Lie algebra representation $\mathrm{ad}^* \colon \mathfrak g \to \mathrm{End}(\mathfrak g^*)$ and is given by $(\mathrm{ad}^*_{y}(\ell))(x) = \ell([y,x])$ for $x,y \in \mathfrak g$ and $\ell \in \mathfrak g^*$.
Let now $G$ be nilpotent. Let $\ell \in \mathfrak g^*$ be fixed, and let $R_{\ell}$ be its stabiliser under the coadjoint action. 
The Lie algebra of $R_{\ell}$ is the \emph{radical} $\mathfrak r_{\ell} = \{y \in \mathfrak g: \ell([y,{}\cdot{}]) = 0\}$.
The coadjoint orbit of $\ell$ is diffeomorphic to $G/R_{\ell}$, and its dimension $\dim \mathfrak g - \dim \mathfrak r_{\ell}$ is always even.
A subspace of $\mathfrak g$ on which $\ell([{}\cdot{},{}\cdot{}])$ vanishes is called \emph{isotropic}.
Maximal isotropic subspaces for $\ell([{}\cdot{},{}\cdot{}])$ have codimension $\tfrac12 (\dim \mathfrak g - \dim \mathfrak r_{\ell})$, and contain $\mathfrak r_{\ell}$.
In our case, they are maximal subalgebras $\mathfrak p_{\ell} \subset \mathfrak g$, and are called \emph{polarizing subalgebras}.
Here are a few key facts:
\begin{enumerate}
\item For any fixed $\ell$, a polarizing subalgebra $\mathfrak p_{\ell}$ always exists.
\item Isotropy ensures that $\ell([\mathfrak p_{\ell},\mathfrak p_{\ell}])=0$, so $\chi_{\ell}(\exp(x)) \coloneqq e^{2\pi i \ell(x)}$ defines a one-dimensional representation of $P_{\ell} \coloneqq \exp(\mathfrak p_{\ell})$.
\item Maximal isotropy ensures that $(P_{\ell},\chi_{\ell})$ \emph{induces} to an irreducible unitary representation of $G$.
\end{enumerate}
The latter induction process works as follows. 
Given a representation $(\pi,\mathcal H_{\pi})$ of a closed subgroup $K \subset G$, one can find a natural representation $\sigma$ (sometimes denoted by $\mathrm{Ind}(K \uparrow G, \pi)$) of $G$ on a new Hilbert space $\mathcal H_{\sigma}$.
Here we describe its \emph{standard model}. 
Consider the set $\mathcal H_{\sigma}$ of all Borel measurable functions $f \colon G \to \mathcal H_{\pi}$ with the following two properties:
\begin{enumerate}
\item covariance along $K$-cosets: $f(kg) = \pi(k)f(g)$, $k \in K$, $g \in G$,
\item $\int_{K \backslash G} \lVert f(g)\rVert^2\,dg < \infty$, where $dg$ is the right-invariant measure on $K \backslash G$.
\end{enumerate}
It turns out that $\mathcal H_{\sigma}$ is complete with respect to the standard $L^2$-scalar product on $K \backslash G$. The \emph{induced representation} $\sigma$ is defined by letting $G$ act on $f \in \mathcal H_{\sigma}$ on the right:
\begin{equation}
\label{eq:right-action-induced-action}
\sigma(g)f \coloneqq f({}\cdot{}g), \qquad g \in G.
\end{equation}

So for any element $\ell \in \mathfrak g^*$ we can construct an irreducible unitary representation of $G$.
The isomorphism class of the representation obtained depends only on the coadjoint orbit of $\ell$, and not on $\ell$ itself or on the choice of a polarizing subalgebra $\mathfrak p_{\ell}$.
In particular, if $\ell, \ell' \in \mathfrak g^*$ lie in the same coadjoint orbit, the induced representations are equivalent.
Also, all irreducible unitary representations of $G$ are obtained from some $\ell \in \mathfrak g^*$ in this way.
Finally, recall that for $G$ nilpotent the exponential map $\exp \colon \mathfrak g \to G$ is a diffeomorphism.
We refer to \cite{corwin-greenleaf} for more details. 

\subsection{Mathai--Quillen currents}
\label{subsec:mathai-quillen-currents}

In order to understand the proper set-up for Mathai--Quillen currents one needs the language of the theory of $\mathbb Z_2$-graded vector spaces.
We now recall the essential definitions we are going to use and set our notations.
We refer to \cite{berline-getzler-vergne} for the general theory.

A $\mathbb Z_2$-graded vector space $V$ is a vector space with a $\mathbb Z_2$-grading $V = V^0 \oplus V^1$.
The degree of $v \in V$ is denoted by $|v|$, and is $0$ (resp.\ $1$) when $v \in V^0$ (resp.\ $V^1$).
A superalgebra $A$ is an algebra whose underlying vector space has the structure of a $\mathbb Z_2$-graded vector space,
and the product respects the grading, i.e.\ $A^i \cdot A^j \subset A^{i+j}$, where indices are taken modulo $2$.
The tensor product of two $\mathbb Z_2$-graded vector spaces $E$ and $F$ is the $\mathbb Z_2$-graded vector space with underlying vector space $E \otimes F$ and grading defined via
\begin{align*}
(E \otimes F)^0 & \coloneqq (E^0 \otimes F^0) \oplus (E^1 \otimes F^1), \\
(E \otimes F)^1 & \coloneqq (E^0 \otimes F^1) \oplus (E^1 \otimes F^0).
\end{align*}
If $E$ and $F$ are superalgebras, the above tensor product is given the structure of superalgebra via $(a_1\otimes b_1)\cdot (a_2\otimes b_2) \coloneqq (-1)^{|b_1||a_2|}a_1a_2 \otimes b_1b_2$.
The tensor superalgebra thus obtained is normally denoted by $E \hatotimes F$. 

Let us discuss Berezin integrals.
Let $E$ and $V$ be two real vector spaces of dimension $n$ and $m$ respectively.
We assume that $E$ is equipped with a scalar product $g^E$, and we fix an orthonormal basis $e_1,\dots,e_n$ of it.
Let $e^1,\dots,e^n$ denote the dual basis.
Further, assume that $E$ is oriented, and that the basis chosen is compatible with the orientation.
Consider the tensor product $\Lambda V^* \hatotimes \Lambda E^*$, and define a linear map $\int^B \colon \Lambda V^* \hatotimes \Lambda E^* \to \Lambda V^*$
such that, for $\alpha \in \Lambda V^*$ and $\beta \in \Lambda E^*$, we have
\begin{enumerate}
\item $\int^B \alpha \cdot \beta = 0$ if $\deg \beta < n$,
\item $\int^B \alpha e^{12\dots n} = \tfrac{(-1)^{n(n+1)/2}}{\pi^{n/2}}\alpha$.
\end{enumerate}
Let $\det E$ be the determinant line of $E$. Then the induced linear map
\[\int^B \colon \Lambda V^* \hatotimes \Lambda E^* \to \Lambda V^* \otimes \det E\]
is called \emph{Berezin integral}.

We now carry over the above algebraic set-up to manifolds. 
The above vector space $V$ will be the model of each tangent space of our manifold, and $E$ be the standard fibre of a vector bundle (a spinor bundle) over the manifold.

Let $M$ be a smooth real manifold of dimension $m$, and let $\pi \colon E \to M$ be a real vector bundle of rank $n$.
We assume $E$ comes equipped with a Riemannian metric $g^E$, and denote by $\nabla^E$ and $R^E$ the corresponding Levi-Civita connection and Riemannian curvature tensor.
We identify $R^E$ with a section of the bundle $\Lambda^2T^*M \hatotimes \Lambda^2E^*$ in the following way.
Let $f_1,\dots,f_m$ be a local frame of $M$ on an open set $U$, and $f^1,\dots,f^m$ be its dual. 
Let $e_1,\dots, e_n$ be a basis of the fibre of $E \to M$, and let $e^1,\dots,e^n$ be the dual basis.
Then $R^E$ is identified with the section 
$\dot{R}^E = \sum_{i<j, \alpha < \beta} g^E(e_{\alpha},R^E(f_i,f_j)e_{\beta}))f^i \wedge f^j \wedge e^{\alpha} \wedge e^{\beta}$.

Let $y \in \Gamma(E,\pi^*E)$ be the tautological section of the pullback bundle $\pi^*E$ over $E$. For $t \geq 0$ define $A_t \coloneqq \pi^*\dot{R}^E+\sqrt t \nabla^{\pi^*E}y+t\lVert y\rVert^2$,  
where $\pi^*\dot{R}^E$ is the pullback of $\dot{R}^E$ to $\pi^*E$, and $\nabla^{\pi^*E}$ the covariant derivative on sections of the pullback bundle.
By \cite[Definition 3.6]{bismut-zhang}, the Mathai--Quillen current is defined as
\begin{equation*}
\psi(\nabla^E,g^E) \coloneqq \int_0^{\infty}\left( \int^B \frac{y}{2\sqrt{t}}\exp(-A_t)\right)\,dt.
\end{equation*}
Let now $\phi \in \Gamma(M,E)$ be a section of $E \to M$, so that $\hat{\phi}\coloneqq \phi^*y$ is now an element in $\Gamma(M,\pi^*E)$.
Then the pullback of the above form is
\begin{equation}
\label{eq:mathai-quillen-form}
\phi^*\psi(\nabla^E,g^E) \coloneqq \int_0^{\infty}\left( \int^B \frac{\hat{\phi}}{2\sqrt{t}}\exp(-(\dot{R}^{E}+\sqrt{t} \nabla^{E}\hat{\phi}+t\lVert \hat{\phi}\rVert^2))\right)\,dt.
\end{equation}

In our case, $E$ will be a spinor bundle over a nilmanifold, and we will be interested in computing this quantity explicitly for any invariant harmonic spinor $\phi$.

\section{Nilmanifolds associated to the Lie algebra \texorpdfstring{$\mathfrak h_1$}{h1}} 
\label{sec:nilmanifold1}

Let $G$ be the  connected, simply connected Lie group with Lie algebra $ \mathfrak g \cong \mathfrak h_1$. Then $\mathfrak g$ has a basis $(E_1,\dots,E_7)$ with structure equations $[E_1,E_2]=-E_6, [E_1,E_3]=-E_7$.
We note that the centre $\mathfrak z(\mathfrak g)$ is four-dimensional, whereas the derived algebra $[\mathfrak g,\mathfrak g] \subset \mathfrak z(\mathfrak g)$ is two-dimensional.
We will use the fact that $G=\exp(\mathfrak g)$.

Let $\Gamma$ be any lattice in $G$. 
By combining Theorem 5 and 6 in \cite{ghorbel-hamrouni}, any such lattice turns out to be isomorphic to one of the following form: let $r=(r_1,r_2)$ be a pair of non-zero natural numbers such that $r_1$ divides $r_2$, then
\begin{equation}
\label{eq:lattices}
\Gamma_r \coloneqq \left\{\exp\left(\frac{m_7}{r_1}E_7\right)\exp\left(\frac{m_6}{r_2}E_6\right)\exp(m_5E_5)\dots \exp(m_1E_1): m_i \in \mathbb Z\right\}.
\end{equation} 
Any two such lattices $\Gamma_r$ and $\Gamma_s$ are isomorphic if and only if $r=s$.
Let now $\Gamma_r$ be any lattice as above, and consider the compact quotient $M\coloneqq \Gamma_r\backslash G$.
This is a compact, connected two-step nilmanifold (in fact, it is a product of a two-torus and a two-torus bundle over a three-torus), and left-invariant data descend to $M$.

The Lie algebra $\mathfrak g$ admits the following family of closed $\mathrm G_2$-structures
\begin{equation}
\label{eq:ex1closedg2}
\varphi_{a} \coloneqq E^{123}+E^{145}+a^2E^{167}+aE^{246}-aE^{257}-aE^{356}-aE^{347},
\end{equation}
depending on one non-zero real number $a$.
This induces the scalar product 
\begin{equation}
\label{eq:left-invariant-metric-general-case1}
g_a = (E^1)^2+\dots+(E^5)^2+a^2(E^6)^2+a^2(E^7)^2, \qquad a \neq 0,
\end{equation}
on $\mathfrak g$ by \eqref{eq:relationship-g2-metric-vol}, as the Riemannian volume form in this case is $a^2E^{12\dots 7}$.
Define $e^i\coloneqq E^i$ for $i=1,\dots,5$, and $e^k\coloneqq aE^k$ for $k=6,7$.
The new structure equations are
\[ d e^j =0, \, j= 1 \ldots, 5, \qquad de^6 = ae^{12}, \qquad de^7 = ae^{13},\]
and in terms of dual vectors we have 
\begin{equation}
\label{eq:lie-alg-str-eq}
[e_1,e_2]=-ae_6, \qquad [e_1,e_3]=-ae_7.
\end{equation}
Then the above three-form $\varphi_a$ can be written as 
\begin{equation}
\label{eq:g2str}
\varphi_a=e^{123}+e^{145}+e^{167}+e^{246}-e^{257}-e^{356}-e^{347}.
\end{equation}
and induces the standard metric $g_a = (e^1)^2+\dots+(e^7)^2$.

Let us collect useful metric data on $G$. Koszul formula yields the Christoffel symbols for the Levi-Civita connection.
Then one can derive the expressions for the spin covariant derivatives and the Dirac operator as in \eqref{def:spin-cov-der} and \eqref{eq:definition-dirac-operator}. 
We will give more precise formulas for the spin covariant derivatives and the Dirac operator only later. In particular we will need to make the action of partial derivatives more explicit.
We will do this by applying Kirillov theory as in Section \ref{subsec:representation-theory-nilpotent-groups} to understand the irreducible representations of $G$.
\begin{lemma}
\label{lemma:covariant-derivatives-dirac}
The Levi-Civita connection $\nabla$ on $(G, g_a)$ is given by
\begin{alignat*}{2}
\nabla_{e_1}e_2 & = -\nabla_{e_2}e_1 = -\tfrac{a}{2} e_6, \qquad && \nabla_{e_1}e_7 = \nabla_{e_7}e_1 = +\tfrac{a}{2} e_3, \\
\nabla_{e_1}e_3 & = -\nabla_{e_3}e_1 = -\tfrac{a}{2} e_7, \qquad && \nabla_{e_2}e_6 = \nabla_{e_6}e_2 = -\tfrac{a}{2} e_1,\\
\nabla_{e_1}e_6 & = +\nabla_{e_6}e_1 = +\tfrac{a}{2} e_2, \qquad && \nabla_{e_3}e_7 = \nabla_{e_7}e_3 = -\tfrac{a}{2} e_1.
\end{alignat*}
The spin covariant derivative $\nabla^S$ on $(G, g_a)$ is given by
\begin{alignat*}{2}
\nabla_{e_1}^S & = \partial_{e_1}-\tfrac{a}{4} \left(e_2e_6+e_3e_7\right), \qquad && \nabla_{e_2}^S = \partial_{e_2}+\tfrac{a}{4} e_1e_6, \\
\nabla_{e_3}^S & = \partial_{e_3}+\tfrac{a}{4} e_1e_7, \qquad && \nabla_{e_k}^S = \partial_{e_k}, \quad k=4,5, \\
\nabla_{e_6}^S & = \partial_{e_6}+\tfrac{a}{4} e_1e_2, \qquad && \nabla_{e_7}^S = \partial_{e_7}+\tfrac{a}{4} e_1e_3.
\end{alignat*}
The spin Dirac operator acting on spinors  defined on $G$ is given by the formula
\begin{equation}
\label{eq:dirac-partial}
D_M = \sum_{k=1}^7 e_k\partial_{e_k}-\tfrac{a}{4}e_1(e_2e_6+e_3e_7).
\end{equation}
\end{lemma}

Recall that the Baker--Campbell--Hausdorff formula \cite{fulton-harris} in the two-step nilpotent case reduces to
\begin{equation}
\label{eq:baker-campbell-hausdorff}
\exp(x)\exp(y)=\exp(x+y)\exp\left(\tfrac12[x,y]\right), \qquad x,y \in \mathfrak g.
\end{equation}

\begin{lemma}
\label{lemma:irreps}
Let $\ell \in \mathfrak g^*$, then we have two cases.
\begin{enumerate}
\item If $\ell$ vanishes on $[\mathfrak g,\mathfrak g]$, then its coadjoint orbit is the single point $\{\ell\}$, and the corresponding irreducible unitary representation is a character.
\item If $\ell$ does not vanish on $[\mathfrak g,\mathfrak g]$, its coadjoint orbit is an affine two-dimensional plane, and the corresponding irreducible unitary representation is infinite-dimensional, with representation space isomorphic to $L^2(\mathbb R,\mathbb C)$.
\end{enumerate}
\end{lemma}
\begin{proof}
Let $\ell = \alpha_1e^1+\dots+\alpha_7e^7$, with $\alpha_i \in \mathbb R$.
Let $\gamma = \exp\left(y_1e_1+\dots+y_7e_7\right) \in G$, and $x=x_1e_1+\dots+x_7e_7 \in \mathfrak g$. The Baker--Campbell--Hausdorff formula \eqref{eq:baker-campbell-hausdorff} gives 
\begin{equation*}
\mathrm{Ad}_{\gamma^{-1}}(x) =\sum_{k=1}^7 x_ke_k+a(y_1x_2-y_2x_1)e_6+a(y_1x_3-y_3x_1)e_7.
\end{equation*}
So if $\ell$ vanishes on $[\mathfrak g,\mathfrak g] = \mathrm{Span}\{e_6,e_7\}$ (i.e.\ $\alpha_6=\alpha_7=0)$, then $\mathrm{Ad}_{\gamma}^*\ell = \ell$.
The bilinear form $\ell([{}\cdot{},{}\cdot{}])$ vanishes identically, and hence the only polarizing subalgebra $\mathfrak p_{\ell}$ is $\mathfrak g$ itself.
The corresponding irreducible unitary representation of $G$ is given by 
\begin{equation}
\label{eq:characters}
\chi_{\ell}(\exp(x)) = e^{2\pi i \ell(x)}, \qquad x \in \mathfrak g.
\end{equation}
If $\ell$ does not vanish on the above hyperplane, then 
\[\mathrm{Ad}_{\gamma}^*\ell = \ell - a(\alpha_6y_2+\alpha_7y_3)e^1+ay_1(\alpha_6e^2+\alpha_7e^3),\]
so the coadjoint orbit of $\ell$ is the affine plane spanned by $e^1$ and $\alpha_6e^2+\alpha_7e^3$.
The radical $\mathfrak r_{\ell}$ is five-dimensional and spanned by $\alpha_7e_2-\alpha_6e_3, e_4,\dots,e_7$,
and a maximal polarizing subalgebra is e.g.\ $\mathfrak p_{\ell} = \mathrm{Span}\{e_2,\dots,e_7\}$.
Note that $\mathfrak p_{\ell}$ is an Abelian ideal of $\mathfrak g$. Set $P_{\ell} \coloneqq \exp(\mathfrak p_{\ell})$, then the map 
\begin{equation*}
\chi_{P_{\ell}}(\exp(y)) = e^{2\pi i \ell(y)}, \qquad y \in \mathfrak p_{\ell},
\end{equation*}
is a one-dimensional representation of $P_{\ell}$. 

We now work out the standard model for the induced representation of $G$. 
By definition, this is realised as a subspace of $L^2(G,\mathbb C)$ in the following way.
Let $f \colon G \to \mathbb C$ be a function in the induced representation of $G$, so it satisfies the equivariance law
\begin{equation}
\label{eq:equivariance1}
f(ph) = \chi_{P_{\ell}}(p)f(h) = e^{2\pi i \ell(\log p)}f(h), \qquad p \in P_{\ell}, h \in G.
\end{equation}
By the Baker--Campbell--Hausdorff formula we obtain
\begin{align*}
\exp\left(\sum_{k=1}^7 x_ke_k\right) = \exp\left(\sum_{k=2}^7 x_ke_k-\tfrac12ax_1(x_2e_6+x_3e_7)\right)\exp(x_1e_1),
\end{align*}
where the first factor in the latter expression sits in $P_{\ell}$, so we write \eqref{eq:equivariance1} as
\begin{equation}
\label{eq:explicit-equivariance1}
f\left(\exp\left(\sum_{k=1}^7 x_ke_k\right)\right) = e^{2\pi i \ell\left(\sum_{k=2}^7x_ke_k-\tfrac12 ax_1(x_2e_6+x_3e_7)\right)}f(\exp(x_1e_1)).
\end{equation}
Set $u_f(t) \coloneqq f(\exp(te_1))$. 
The map $f \mapsto u_f$ sets up a linear bijection between the subspace of $L^2(G,\mathbb C)$ of functions satisfying \eqref{eq:equivariance1} and $L^2(\mathbb R,\mathbb C)$.
There remains to understand how $G$ acts on $u_f$.
In order to do this, one first understands the action $\sigma$ of $G$ on the corresponding $f$, cf.\ \eqref{eq:right-action-induced-action}.
Take $h = \exp(y_1e_1+\dots+y_7e_7)$ and $g = \exp(x_1e_1+\dots+x_7e_7)$. 
The action $\sigma$ was defined as $(\sigma(h)f)(g) = f(gh)$. 
The action on $u_f$ is obtained by setting $x_2=\dots=x_7=0$. Set $t \coloneqq x_1$, then applying \eqref{eq:explicit-equivariance1} yields
\begin{equation}
\label{eq:infinite-dimensional-irrep1}
(\rho_{\ell}(h)u_f)(t) \coloneqq e^{2\pi i \left(\sum_{k=2}^7\alpha_ky_k-a\left(t+\tfrac{y_1}{2}\right)(\alpha_6y_2+\alpha_7y_3)\right)}u_f(t+y_1).
\end{equation}
This defines the action of $G$ on $L^2(\mathbb R,\mathbb C)$.
\end{proof}
The above representation-theoretic considerations allow us to split the space $L^2(SM)$ of $L^2$-spinors defined on the nilmanifold $M= \Gamma_r \backslash G$ into simpler summands.
Note that in this process, the left-invariant metric on $G$ does not play any role, and up to isomorphism the lattice can be chosen of the form $\Gamma_r$ as in \eqref{eq:lattices}.
We always work with the trivial spin structure from now on.

Let us define the right regular action $R$ of $G$ on $L^2(SM)$.
Elements of $L^2(SM)$ can be identified with $\Gamma_r$-invariant $L^2$-spinors on $G$.
Let $\sigma$ be an element of $L^2(SM)$, then $R$ is defined via
\[R(h)\sigma \coloneqq \sigma({}\cdot{}h), \qquad h \in G.\]
The differential of this action is
\begin{equation}
\label{eq:differential-right-regular-action}
R_*(x)\sigma = \frac{d}{dt}\left(R(\exp(tx)\sigma)\right)\Bigr|_{\substack{t=0}} \eqqcolon \partial_x\sigma, \qquad x \in \mathfrak g,
\end{equation}
where $x$ in the symbol $\partial_x$ is the vector field induced by $x \in \mathfrak g$ on $M$.
For any fixed $\gamma \in G$ and $\Gamma_r$-invariant $L^2$-spinor $\sigma$ on $G$, the function $\varphi_{\gamma} \colon \mathbb R^4 \to \Delta$ such that 
\[(y_4,\dots,y_7) \mapsto \left(R\left(\exp\left(y_4e_4+y_5e_5+\frac{ay_6}{r_2}e_6+\frac{ay_7}{r_1}e_7\right)\right)\sigma\right)(\gamma)\]
is $1$-periodic in each variable. Therefore, we can write $\varphi_{\gamma}$ as a Fourier series
\[\varphi_{\gamma}(y_4,\dots,y_7) = \sum_{\alpha \in \mathbb Z^4} \varphi_{\alpha}(\gamma)e^{2\pi i \langle \alpha,y\rangle},\]
where $\alpha=(\alpha_4,\dots,\alpha_7) \in \mathbb Z^4$, $y=(y_4,\dots,y_7)$, $\langle \alpha,y\rangle = \sum_{k=4}^7 \alpha_ky_k$, and the coefficient functions $\varphi_{\alpha}$ are given by
\[\varphi_{\alpha}(\gamma) = \int_{[0,1]^4} \varphi_{\gamma}(y_4,\dots,y_7) e^{-2\pi i \langle \alpha,y\rangle}\,dy_4\dots dy_7.\]
Note that $\sigma(\gamma) = \varphi_{\gamma}(0,0,0,0) = \sum_{\alpha \in \mathbb Z^4} \varphi_{\alpha}(\gamma)$, 
which gives a first $G$-invariant decomposition
\[L^2(SM) = \bigoplus_{\alpha \in \mathbb Z^4} H_{\alpha} = H_0 \oplus \bigoplus_{\alpha \in \mathbb Z^4 \setminus \{0\}} H_{\alpha}.\]
We will see that $H_0$ contains the space of left-invariant harmonic spinors.
A simple computation yields that for any element of the form $\gamma(z_4,\dots,z_7)\coloneqq \exp(z_4e_4+z_5e_5+\frac{a}{r_2}z_6e_6+\frac{a}{r_1}z_7e_7)$ in the centre of $G$, we have
\begin{equation}
\label{eq:action-ha}
R(\gamma(z_4,\dots,z_7))\varphi_{\alpha} = e^{2\pi i \langle \alpha,z\rangle}\varphi_{\alpha}.
\end{equation}
We then need to understand how the remaining elements of $G$ act on each $H_{\alpha}$.
\begin{lemma}
\label{lemma:decomposition-l2}
Each $H_{(\alpha_4,\dots,\alpha_7)}$ decomposes as a direct sum of $G$-invariant summands in the following way. Let $\alpha=(\alpha_4,\dots,\alpha_7)$.
\begin{enumerate}
\item If $\alpha_6^2+\alpha_7^2=0$, then $H_{\alpha}$ splits as a direct sum of subspaces isomorphic to $\mathbb C \otimes_{\mathbb C} \Delta \simeq \Delta$.
Here $G$ acts on the first factor via $\chi_{\ell}$ \eqref{eq:characters}, for 
\[\ell = \sum_{k=1}^5 \alpha_k e^k \in \mathfrak g^*, \qquad \alpha_1,\dots,\alpha_5 \in \mathbb Z.\]
\item If $\alpha_6^2+\alpha_7^2>0$, then $H_{\alpha}$ splits as a finite direct sum of subspaces isomorphic to $L^2(\mathbb R,\mathbb C) \otimes_{\mathbb C} \Delta$.
Here $G$ acts on the first factor via $\rho_{\ell}$ \eqref{eq:infinite-dimensional-irrep1}, for 
\[\ell = \sum_{k=2}^5 \alpha_k e^k + \frac{r_2\alpha_6}{a}e^6+\frac{r_1\alpha_7}{a}e^7 \in \mathfrak g^*.\]
\end{enumerate}
In both cases, the $G$-action commutes with the action of the complex Clifford algebra $\mathrm{Cl}^{\mathbb C}(7) = \mathrm{Cl}(7) \otimes \mathbb C$ on $\Delta$.
\end{lemma}
\begin{proof}
Let $\sigma \in H_{\alpha}$ and $x=\sum_{k=1}^7x_ke_k$. By \eqref{eq:action-ha}, we can write 
\[\sigma(\exp(x)) = e^{2\pi i (\sum_{k=4}^7 \alpha_kx_k)}\sigma(\exp(x_1e_1+x_2e_2+x_3e_3)).\]
If $\alpha_6^2+\alpha_7^2=0$, the commutator group $[G,G]$ acts trivially, so we can write
\[\sigma(\exp(x)) = e^{2\pi i (\sum_{k=4}^5 \alpha_kx_k)}\sigma(\exp(x_1e_1)\exp(x_2e_2)\exp(x_3e_3)).\]
The fuction $(x_1,x_2,x_3) \mapsto \sigma(\exp(x_1e_1)\exp(x_2e_2)\exp(x_3e_3))$ is $1$-periodic in each variable, so we have a Fourier expansion
\[\sigma(\exp(x)) =e^{2\pi i (\alpha_4x_4+\alpha_5x_5)} \sum_{\alpha_1,\alpha_2,\alpha_3 \in \mathbb Z} a_{\alpha_1,\alpha_2,\alpha_3}e^{2\pi i (\sum_{k=1}^3 \alpha_kx_k)},\]
where each $a_{\alpha_1,\alpha_2,\alpha_3}$ is a constant spinor. 
We then have a splitting
\[H_{(\alpha_4,\alpha_5,0,0)} = \bigoplus_{\alpha_1,\alpha_2,\alpha_3 \in \mathbb Z} H_{(\alpha_1,\dots,\alpha_5,0,0)},\]
where each summand is isomorphic to $\mathbb C \otimes_{\mathbb C} \Delta = \Delta$, and $G$ acts on the first factor via multiplication by $e^{2\pi i (\alpha_1x_1+\dots+\alpha_5x_5)}$. 
Clearly this action commutes with the action of the Clifford algebra $\mathrm{Cl}^{\mathbb C}(7)$ on $\Delta$. 
By Lemma \ref{lemma:irreps}, this is the decomposition into irreducible summands.
Note that $H_{(\alpha_1,\dots,\alpha_5,0,0)} = H_{(0,\dots,0)}$ is the space of left-invariant harmonic spinors, as $G$ acts trivially on it.

Let us look at the case $\alpha_6^2+\alpha_7^2>0$. 
Let $\sigma \in H_{\alpha}$ and $x=\sum_{k=1}^7 x_ke_k$. By \eqref{eq:action-ha} and the Baker--Campbell--Hausdorff formula we have
\begin{align*}
\sigma(\exp(x)) & = e^{2\pi i \left(\alpha_4x_4+\alpha_5x_5+\frac{r_2\alpha_6}{a}x_6+\frac{r_1\alpha_7}{a}x_7\right)}\sigma(\exp(x_1e_1+x_2e_2+x_3e_3)),
\end{align*}
and by similar steps as above
\begin{align*}
\sigma(\exp(x_1e_1+x_2e_2+x_3e_3)) & = e^{-\pi i x_1(r_2x_2\alpha_6+r_1x_3\alpha_7)}\sigma(\exp(x_2e_2+x_3e_3)\exp(x_1e_1)) \\
& \eqqcolon e^{-\pi i x_1(r_2x_2\alpha_6+r_1x_3\alpha_7)}\varphi_{x_1}(x_2,x_3).
\end{align*}
Let $x_1$ be fixed. Then $\varphi_{x_1}(x_2,x_3)$ is $1$-periodic in each variable, and hence we can write it as
\[\varphi_{x_1}(x_2,x_3) = \sum_{\alpha_2,\alpha_3 \in \mathbb Z} \varphi_{\alpha_2,\alpha_3}(x_1)e^{2\pi i (\alpha_2x_2+\alpha_3x_3)},\]
with
\[\varphi_{\alpha_2,\alpha_3}(x_1) = \int_{[0,1]^2} \varphi_{x_1}(y_2,y_3)e^{-2\pi i (\alpha_2y_2+\alpha_3y_3)}\,dy_2dy_3.\]
Therefore, $\sigma(\exp(x))$ equals
\[
e^{2\pi i \left(\sum_{k=4}^5 \alpha_kx_k+r_2\alpha_6\left(\frac{x_6}{a}-\frac{x_1x_2}{2}\right)+r_1\alpha_7\left(\frac{x_7}{a}-\frac{x_1x_3}{2}\right)\right)} \sum_{\alpha_2,\alpha_3 \in \mathbb Z}\varphi_{\alpha_2,\alpha_3}(x_1)e^{2\pi i (\alpha_2x_2+\alpha_3x_3)}.
\]
By using $\Gamma$-invariance in the $x_1$-variable, one checks the relations
\[\varphi_{\alpha_2,\alpha_3}(x_1-m) = \varphi_{\alpha_2+\alpha_6mr_2,\alpha_3+\alpha_7mr_1}(x_1), \qquad m \in \mathbb Z.\]
So we have $r_1r_2\alpha_6\alpha_7$ independent functions, and consequently an isomorphism 
\[H_{(\alpha_4,\dots,\alpha_7)} = \bigoplus_{k=0}^{\alpha_6r_2-1}\bigoplus_{h=0}^{\alpha_7r_1-1} H_{(\alpha_2+k,\alpha_3+h,\alpha_4,\dots,\alpha_7)},\]
where each summand is isomorphic to $L^2(\mathbb R,\mathbb C) \otimes_{\mathbb C} \Delta$. 
Lastly, we need to understand how $G$ acts on each function $\varphi_{\alpha_2,\alpha_3}$. 
In order to do this, we first understand the right regular action of $h=\exp(y_1e_1+\dots+y_7e_7)$ on $\sigma(\exp(x))$,
then restrict the action to $\varphi_{\alpha_2,\alpha_3}$ by setting $x_2=\dots=x_7=0$. Set $t=x_1$, an explicit computation yields
\begin{align*}
(h\cdot \varphi_{\alpha_2,\alpha_3})(t) & = e^{2\pi i \left(\sum_{k=2}^5 \alpha_ky_k+\frac{r_2\alpha_6}{a}y_6+\frac{r_1\alpha_7}{a}y_7-a\left(t+\frac{y_1}{2}\right)(\frac{r_2\alpha_6}{a}y_2+\frac{r_1\alpha_7}{a}y_3)\right)} \\
& \qquad \cdot \varphi_{\alpha_2,\alpha_3}(t+y_1).
\end{align*}
Therefore, $G$ acts on the first factor of $H_{(\alpha_2,\dots,\alpha_7)} = L^2(\mathbb R,\mathbb C) \otimes_{\mathbb C} \Delta$ via $\rho_{\ell}$ \eqref{eq:infinite-dimensional-irrep1} with \[\ell = \sum_{k=2}^5 \alpha_ke^k+\frac{r_2\alpha_6}{a}e^6+\frac{r_1\alpha_7}{a}e^7.\]
Again, this action commutes with the action of the Clifford algebra on $\Delta$.
\end{proof}
We are ready to study the kernel of the spin Dirac operator $D_M$ acting on $L^2(SM)$. 
Lemma~\ref{lemma:decomposition-l2} tells us that $D_M$ acts on each irreducible summand in the decomposition of $H_{(\alpha_4,\dots,\alpha_7)}$, so it is enough to study the action of $D_M$ on these irreducible representations. 
We can then use the results in Lemma \ref{lemma:covariant-derivatives-dirac}, \ref{lemma:irreps}, and \ref{lemma:decomposition-l2} for our computations.
\begin{proposition}
\label{prop:harmonic-spinors}
Let $G$ be the connected, simply connected two-step nilpotent Lie group with Lie algebra isomorphic to $\mathfrak h_1$ and $M = \Gamma \backslash G$ be any associated nilmanifold.
For every Riemannian metric as in \eqref{eq:left-invariant-metric-general-case1} and the trivial spin structure, the space of harmonic spinors for the spin Dirac operator $D_M$ is even-dimensional.
Moreover, there exist non-zero, non-invariant harmonic spinors.
\end{proposition}
\begin{proof} 
We first study the action of $D_M$ on invariant spinors, then on the irreducible $G$-invariant summands of each $H_{(\alpha_4,\dots,\alpha_7)}$ in Lemma \ref{lemma:decomposition-l2}.

Let $\phi$ be the only harmonic left-invariant unit spinor compatible with the $\mathrm G_2$-structure \eqref{eq:g2str} (uniquely defined up to a sign).
By the compatibility relation \eqref{eq:three-form-spinor} and the three-form \eqref{eq:g2str}, one computes each $e_ie_j\phi$.
The intrinsic endomorphism $S$ (see \eqref{eq:intrinsic-endomorphism}) for $(M,g,\phi)$ is easily seen to be given by 
\begin{alignat}{2}
\label{eq:intrinsic-end1}
S(e_2) & = -\tfrac{a}{4} e_7, && \qquad S(e_3) = +\tfrac{a}{4} e_6, \\
\label{eq:intrinsic-end2}
S(e_6) & = -\tfrac{a}{4} e_3, && \qquad S(e_7) = +\tfrac{a}{4} e_2,
\end{alignat}
and $S(e_1)=S(e_4)=S(e_5)=0$. Recall our expression for the Dirac operator \eqref{eq:dirac-partial}.
The invariant spinors $\phi$, $e_1\phi$, $e_4\phi$, $e_5\phi$ are clearly harmonic, and we compute
\begin{alignat*}{2}
D_M(e_2\phi) & = +\tfrac{a}{2} e_7\phi, && \qquad D_M(e_3\phi) = -\tfrac{a}{2}e_6\phi, \\
D_M(e_6\phi) & = -\tfrac{a}{2} e_3\phi, && \qquad D_M(e_7\phi) = +\tfrac{a}{2} e_2\phi.
\end{alignat*}

We now consider non-invariant spinors.
Let us compute the action of $D_M$ on each irreducible summand of $H_{(\alpha_4,\alpha_5,0,0)}$. 
Recall that $H_{(\alpha_4,\alpha_5,0,0)}$ splits as the direct sum of spaces $H_{(\alpha_1,\dots,\alpha_5,0,0)} \simeq \mathbb C \otimes_{\mathbb C} \Delta \simeq \Delta$, and $\exp(\sum_k x_ke_k) \in G$ acts on the first factor of each of them via multiplication by $e^{2\pi i (\alpha_1x_1+\dots+\alpha_5x_5)}$.
In this case, the differential of the right regular action \eqref{eq:differential-right-regular-action} is multiplication by $2\pi i \alpha_k$ for $k=1,\dots,5$, and zero for $k=6,7$.
Define $\beta_k \coloneqq 2\pi i \alpha_k$ for $k=1,\dots,5$.
Set $e_0 \coloneqq 1$ for convenience, and recall that $\{e_0\phi,\dots,e_7\phi\}$ is a basis of $\Delta$. Then \eqref{eq:dirac-partial} gives
\begin{align*}
D_M(e_0\phi) & = \beta_1e_1\phi+\beta_2e_2\phi+\beta_3e_3\phi+\beta_4e_4\phi+\beta_5e_5\phi, \\
D_M(e_1\phi) & = \overline \beta_1\phi+\overline\beta_3e_2\phi+\beta_2e_3\phi+\overline\beta_5e_4\phi+\beta_4e_5\phi, \\
D_M(e_2\phi) & = \overline \beta_2\phi+\beta_3e_1\phi+\overline \beta_1e_3\phi+\beta_4e_6\phi+\left(\overline \beta_5+\tfrac{a}{2}\right)e_7\phi, \\
D_M(e_3\phi) & = \overline \beta_3\phi+\overline \beta_2e_1\phi+\beta_1e_2\phi+\left(\overline \beta_5-\tfrac{a}{2}\right)e_6\phi+\overline \beta_4e_7\phi, \\
D_M(e_4\phi) & = \overline \beta_4\phi+\beta_5e_1\phi+\overline \beta_1e_5\phi+\overline \beta_2e_6\phi+\beta_3e_7\phi, \\
D_M(e_5\phi) & = \overline \beta_5\phi+\overline \beta_4e_1\phi+\beta_1e_4\phi+\beta_3e_6\phi+\beta_2e_7\phi, 
\end{align*}
\begin{align*}
D_M(e_6\phi) & = \overline \beta_4e_2\phi+\left(\beta_5-\tfrac{a}{2}\right)e_3\phi+\beta_2e_4\phi+\overline \beta_3e_5\phi+\overline \beta_1e_7\phi, \\
D_M(e_7\phi) & = \left(\beta_5+\tfrac{a}{2}\right)e_2\phi+\beta_4e_3\phi+\overline \beta_3e_4\phi+\overline \beta_2e_5\phi+\beta_1e_6\phi.
\end{align*}
Note that if $\beta_k=0$ for all $k=1,\dots,5$, i.e.\ we look at the action of $D_M$ on $H_{(0,\dots,0)}$, we recover the results on the invariant spinors, so we can assume at least one $\beta_i$ to be non-zero. 
Let $D$ be the $8\times 8$ matrix whose columns are the vectors $D_M(e_i\phi)$'s, for $i=0,\dots,7$.
We can split $D$ as $D=D_H+D_S$, where $D_H$ is the part of $D$ whose entries are the $\beta_k$'s, and $D_S$ is the remaining real part.
Note that $D_H^2 = \mu^2 \mathrm{id}$, where $\mu^2 = \sum_{k=1}^5 |\beta_k|^2>0$.
If $\psi \in \ker D \subset \Delta$, then $D_H\psi=-D_S\psi$, so $\mu^2\psi=D_H^2\psi=-D_HD_S\psi$. 
Write $\psi = \sum_{k=0}^7 \psi^ke_k\phi$.
Then $D_HD_S\psi = -\mu^2\psi$ is equivalent to four linear systems in matrix form
\begin{align*}
\mu^2\begin{pmatrix} \psi_0 \\ \psi_1 \end{pmatrix} & = \frac{a}{2}\begin{pmatrix} \overline \beta_3 & \beta_2 \\ \overline \beta_2 & \overline\beta_3 \end{pmatrix}\begin{pmatrix} \psi_6 \\ \psi_7 \end{pmatrix}, \\
\mu^2\begin{pmatrix} \psi_2 \\ \psi_3 \end{pmatrix} & = \frac{a}{2}\begin{pmatrix} \overline \beta_5 & \overline \beta_4 \\ \overline \beta_4 & \beta_5 \end{pmatrix}\begin{pmatrix} \psi_2 \\ \psi_3 \end{pmatrix} + \frac{a\beta_1}{2}\begin{pmatrix} \psi_6 \\ \psi_7 \end{pmatrix}, \\
\mu^2\begin{pmatrix} \psi_4 \\ \psi_5 \end{pmatrix} & = \frac{a}{2}\begin{pmatrix} \beta_3 & \beta_2 \\ \beta_2 & \overline\beta_3 \end{pmatrix}\begin{pmatrix} \psi_2 \\ \psi_3 \end{pmatrix}, \\
\mu^2\begin{pmatrix} \psi_6 \\ \psi_7 \end{pmatrix} & = \frac{a}{2}\begin{pmatrix} \overline \beta_5 & \overline \beta_4 \\ \overline \beta_4 & \beta_5 \end{pmatrix}\begin{pmatrix} \psi_6 \\ \psi_7 \end{pmatrix} + \frac{a\overline \beta_1}{2}\begin{pmatrix} \psi_2 \\ \psi_3 \end{pmatrix}.
\end{align*}
The case $\beta_1=0$ is trivial, and shows that $\psi=0$, so can assume $\beta_1 \neq 0$.
Combining second and fourth system we find the equation
\[\left(A^2+c^2\mathrm{id}_2\right)\begin{pmatrix} \psi_2 \\ \psi_3\end{pmatrix} = 0, \quad \text{with } c = \frac{a\beta_1}{2}, \quad A = \begin{pmatrix} \frac{a\beta_5}{2}+\mu^2 & \frac{a\beta_4}{2} \\ \frac{a\beta_4}{2} & \mu^2+\frac{a\overline\beta_5}{2}\end{pmatrix}.\]
A direct computation shows that $\det(A^2+c^2\mathrm{id}_2)=0$ if and only if $\beta_4=\beta_5=0$ and $|\beta_1|^2+|\beta_2|^2+|\beta_3|^2=\frac{|a|}{2}|\beta_1|$.
So if $\det(A^2+c^2\mathrm{id})\neq 0$, the pair $(\psi_2,\psi_3)$ vanishes, and hence $\psi=0$ as $\mu^2>0$.
If $\det(A^2+c^2\mathrm{id})=0$, then $A^2+c^2\mathrm{id}=0$, and hence $(\psi_2,\psi_3)$ is free and determines $\psi$ by the above system.
Note that the condition $|\beta_1|^2+|\beta_2|^2+|\beta_3|^2=\frac{|a|}{2}|\beta_1|$ is equivalent to having
\[4\pi (\alpha_1^2+\alpha_2^2+\alpha_3^2) = |a||\alpha_1|.\]
In order to have non-trivial solutions, we must have $|a|=k\pi$, for some integer $k >0$. 
But then 
\[\frac{k}{4}|\alpha_1|-\alpha_1^2 = \alpha_2^2+\alpha_3^2\geq 0,\]
and this forces $\alpha_1$ to range in a finite set, as $k$ (equivalently $|a|$), is fixed a priori.
This implies $\alpha_2,\alpha_3$ range in a finite set as well.
So for each suitable choice of the parameters, $(\psi_2,\psi_3)$ determines $\psi$, and we have an even-dimensional space of harmonic spinors.

Now for the action of $D_M$ on each summand of $H_{(\alpha_4,\dots,\alpha_7)}$ when $\alpha_6^2+\alpha_7^2>0$. 
In this case each irreducible summand in the decomposition of $H_{(\alpha_4,\dots,\alpha_7)}$ is isomorphic to $L^2(\mathbb R,\mathbb C) \otimes_{\mathbb C} \Delta$, so in order to write $D_M$ in matrix form we use an explicit basis of $L^2(\mathbb R,\mathbb C)$. By Lemma \ref{lemma:decomposition-l2} and formula \eqref{eq:dirac-partial}, we compute 
\begin{align*}
D_M & = \frac{d}{dt}e_1+(\beta_2-r_2\beta_6t)e_2+(\beta_3-r_1\beta_7t)e_3+\beta_4e_4+\beta_5e_5 \\
& \qquad +\frac{r_2\beta_6}{a}e_6+\frac{r_1\beta_7}{a}e_7 -\frac{a}{4}e_1(e_2e_6+e_3e_7).
\end{align*}
The space $L^2(\mathbb R,\mathbb C)$ has an orthogonal Schauder basis of Hermite functions 
\[h_k(t) \coloneqq e^{\frac{t^2}{2}}\left(\frac{d}{dt}\right)^ke^{-t^2}, \qquad k \in \mathbb N,\]
which satisfy the relations 
\begin{align*}
h_k'(t) & = th_k(t)+h_{k+1}(t), \\
0 & = h_{k+2}(t)+2th_{k+1}(t)+2(k+1)h_k(t),
\end{align*}
for all $k \in \mathbb N$. 
These can be rewritten dropping the dependence of $h_k$ on $t$ as
\[th_k = -\frac12 h_{k+1}-kh_{k-1}, \qquad h_k' = +\frac12 h_{k+1}-kh_{k-1}.\]
Set $h_{-1}(t)\coloneqq 0$.
Now, for a fixed $k$ compute $D_M(h_k \otimes e_i\phi)$ for all $i=0,\dots,7$.
It turns out that each $D_M(h_k \otimes e_i\phi)$ lies in the span of $h_{k-1} \otimes e_i\phi, h_k \otimes e_i\phi$, and $h_{k+1} \otimes e_i\phi$ for $i=0,\dots,7$.
The operator $D_M$ can then be represented by an infinite matrix $D$ of the form 
\[
D = \left(
\begin{matrix}
M & N & & & & \dots \\
-\frac12 \overline N & M & 2N & & & \dots \\
& -\frac12 \overline N & M & 3N & & \dots \\
& & -\frac12 \overline N & M & 4N & \dots \\
& & & -\frac12 \overline N & M & \dots \\
\vdots & \vdots & \vdots & \vdots & \vdots & \ddots
\end{matrix}
\right)
\]
where $M$ and $N$ are the following non-vanishing $8 \times 8$ matrices:
\begin{align*}
M & = \left(
\begin{matrix}
0 & 0 & \overline \beta_2 & \overline \beta_3 & \overline \beta_4 & \overline \beta_5 & \frac{r_2\overline \beta_6}{a} & \frac{r_1\overline \beta_7}{a} \\
0 & 0 & \beta_3 & \overline \beta_2 & \beta_5 & \overline \beta_4 & \frac{r_1\beta_7}{a} & \frac{r_2\overline \beta_6}{a} \\
\beta_2 & \overline \beta_3 & 0 & 0 & \frac{r_2\beta_6}{a} & \frac{r_1\overline \beta_7}{a} & \overline \beta_4 & \beta_5+\frac{a}{2} \\
\beta_3 & \beta_2 & 0 & 0 & \frac{r_1\overline \beta_7}{a} & \frac{r_2\overline \beta_6}{a} & \beta_5-\frac{a}{2} & \beta_4 \\
\beta_4 & \overline \beta_5 & \frac{r_2\overline \beta_6}{a} & \frac{r_1\beta_7}{a} & 0 & 0 & \beta_2 & \overline \beta_3 \\
\beta_5 & \beta_4 & \frac{r_1\beta_7}{a} & \frac{r_2\beta_6}{a} & 0 & 0 & \overline \beta_3 & \overline \beta_2 \\
\frac{r_2\beta_6}{a} & \frac{r_1\overline \beta_7}{a} & \beta_4 & \overline \beta_5-\frac{a}{2} & \overline \beta_2 & \beta_3 & 0 & 0 \\
\frac{r_1\beta_7}{a} & \frac{r_2\beta_6}{a} & \overline \beta_5+\frac{a}{2} & \overline \beta_4 & \beta_3 & \beta_2 & 0 & 0 
\end{matrix}
\right), \\
N & = \left(
\begin{matrix}
0 & 1 & r_2\overline \beta_6 & r_1\overline \beta_7 & 0 & 0 & 0 & 0 \\
-1 & 0 & r_1\beta_7 & r_2\overline \beta_6 & 0 & 0 & 0 & 0 \\
r_2\beta_6 & r_1\overline \beta_7 & 0 & -1 & 0 & 0 & 0 & 0 \\
r_1\beta_7 & r_2\beta_6 & 1 & 0 & 0 & 0 & 0 & 0 \\
0 & 0 & 0 & 0 & 0 & -1 & r_2\beta_6 & r_1\overline \beta_7 \\
0 & 0 & 0 & 0 & 1 & 0 & r_1\overline \beta_7 & r_2\overline \beta_6 \\
0 & 0 & 0 & 0 & r_2\overline \beta_6 & r_1\beta_7 & 0 & -1 \\
0 & 0 & 0 & 0 & r_1\beta_7 & r_2\beta_6 & 1 & 0 
\end{matrix}
\right).
\end{align*}
Any element $\psi = \sum_{k \in \mathbb N} \sum_{i=0}^7 \psi_k^ih_k \otimes e_i\phi \in L^2(\mathbb R,\mathbb C) \otimes \Delta$ can be written as a sum $\psi = \sum_{k \in \mathbb N} h_k \otimes \psi_k$, where $\psi_k = \sum_{i=0}^7 \psi_k^ie_i\phi \in \Delta$. 
Also, for all $i=0,\dots,7$ we have $\sum_{k=0}^{\infty}|\psi_k^i|^2 <\infty$, as the functions $\sum_{k\in\mathbb N} \psi_k^ih_k$ are $L^2$-integrable.
Then $\psi$ is harmonic if and only if
\begin{align*}
0 & = M\psi_0+N\psi_1, \\
0 & = -\frac12 \overline N\psi_k+M\psi_{k+1}+(k+2)N\psi_{k+2}, \qquad k \geq 0.
\end{align*}
Note that $\det N = ((r_2^2|\beta_6|^2+r_1^2|\beta_7|^2)^2-1)^4$ is positive.
Also, $N^2 = \mu^2\mathrm{id}_8$, with $\mu^2 = r_1^2|\beta_7|^2+r_2^2|\beta_6|^2-1>0$.
Then the above system implies that $\psi$ is uniquely determined by $\psi_0$ via
\begin{align*}
\psi_1& =-\frac{1}{\mu^2}NM\psi_0, \\
\psi_{k+2}& =-\frac{1}{\mu^2(k+2)}N\left(-\frac12\overline N\psi_k+M\psi_{k+1}\right), \qquad k \geq 0.
\end{align*}
Note that the sequence $(\psi_k)_{k=0}^{\infty}$ gives a harmonic spinor if and only if it satisfies the above recursive system and
\begin{equation}
\label{eq:l2}
\sum_{k=0}^{\infty}\lVert \psi_k\rVert^2 <\infty.
\end{equation}
Let $A\coloneqq \max\{\lVert N\overline N\rVert/2\mu^2, \lVert NM\rVert/\mu^2\}$.
Up to taking an equivalent norm on $\Delta$, we can assume $2A\leq1$.
Choose any non-zero element $\psi_0 \in \Delta$ and consider the corresponding sequence $(\psi_k)_{k=0}^{\infty}$ obtained by the above recursion. 
We show that there is a constant $C>0$ depending on $\beta_2,\dots,\beta_7$ and $\lVert \psi_0\rVert$ such that 
\begin{equation}
\label{eq:estimate}
\lVert \psi_k\rVert \leq \frac{C}{k+1}, \qquad k \geq 0.
\end{equation}
For finitely many values $0,1,\dots,k+1$ the inequality \eqref{eq:estimate} is certainly satisfied for some constant $C$ large enough.
Assume now 
\[\lVert \psi_k\rVert \leq \frac{C}{k+1}, \qquad \lVert \psi_{k+1}\rVert \leq \frac{C}{k+2}\]
for some $C$.
By inductive hypothesis we find
\begin{align*}
\lVert\psi_{k+2}\rVert & \leq \frac{A}{k+2}(\lVert \psi_k\rVert+\lVert \psi_{k+1}\rVert) \leq \frac{AC}{k+2}\left(\frac{1}{k+1}+\frac{1}{k+2}\right) \\
& \leq \frac{2AC}{(k+2)(k+1)} \leq \frac{C}{k+3}.
\end{align*}
The latter inequality is true for $k \geq 1$. 
The claim is proved.
This shows that the kernel of the Dirac operator acting on each $H_{(\alpha_2,\dots,\alpha_7)}$ is eight-dimensional.
Since we have finitely many irreducible spaces $H_{(\alpha_2,\dots,\alpha_7)}$, the kernel of $D_M$ is even-dimensional.
\end{proof}

\subsection{Eta invariants}
\label{sec:eta-invariants}

We noted that the existence of an orientation-reversing isometry on $(M,g,\varphi,\phi)$ implies the vanishing of the $\eta$-invariants of $D_M$ and $B_M$ by \cite[Proposition 1.5]{crowley-goette-nordstrom}.
Such an isometry is given below.
\begin{proposition}
\label{prop:odd-signature}
Let $G$ be the connected, simply connected two-step nilpotent Lie group with Lie algebra isomorphic to $\mathfrak h_1$ endowed with a left-invariant closed $G_2$-structure $\varphi_a$ \eqref{eq:g2str}.
Let $M=\Gamma\backslash G$ be any nilmanifold obtained from $G$ with metric $g_a$ induced by $\varphi_a$.
Then $(M,g_a)$ admits an orientation-reversing isometry.
\end{proposition}
\begin{proof}
Let $\mathfrak g \cong \mathfrak h_1$ be the Lie algebra of $G$. We have noted that $\mathfrak z(\mathfrak g) = \mathbb R e_4 \oplus \mathbb R e_5 \oplus [\mathfrak g,\mathfrak g]$.
Observe that any lattice $\Gamma$ in $G$ splits as $\Gamma' \times \mathbb Z^2$, where $\Gamma'$ is a lattice in $\exp(\mathrm{Span}_{\mathbb R}\{e_1,e_2,e_3,e_6,e_7\})$, and $\mathbb Z^2 = \mathbb Z e_4 \oplus \mathbb Ze_5$ \cite[Theorem 5]{ghorbel-hamrouni}.
So let us fix any $\Gamma$.
Define a linear map $\tilde T \colon \mathfrak g \mapsto \mathfrak g$ such that $\tilde T e_4 = -e_4$, and $\tilde T e_i = e_i$ for all $i\neq 4$.
Then $\tilde T$ is an orientation-reversing isometry, and $\tilde T(\log \Gamma) = \log \Gamma$.
Then the map 
\[T \coloneqq \exp \circ\ \tilde T \circ \log \colon G \to G\]
preserves $\Gamma$, and thus descends to an orientation-reversing isometry on $\Gamma\backslash G$. 
\end{proof}

\subsection{The Mathai--Quillen current}
\label{sec:mathai-quillen-current}

Recall the set-up in Section \ref{subsec:mathai-quillen-currents}.
In our case $M$ is a nilmanifold $\Gamma \backslash G$, and $E$ is the spin bundle $SM$ with its metric and spinorial data specified above, particularly in Lemma \ref{lemma:covariant-derivatives-dirac}.
Let $\phi$ be the invariant unit spinor compatible with the closed $\mathrm G_2$-structure $\varphi_a$ \eqref{eq:g2str} on $M$, and let $g = (e^1)^2+\dots+(e^7)^2$ the Riemannian metric induced by $\varphi$.
\begin{proposition}
\label{prop:mathai-quillen-current-vanishes}
The Mathai--Quillen current on $(M,g,\phi)$ vanishes.
\end{proposition}
\begin{proof}
Recall that the Mathai--Quillen current was defined in \eqref{eq:mathai-quillen-form}. 
By a change of variable $t \mapsto t^2$, we see that we need to compute
\[\int_M \phi^*\psi(\nabla^{S},g^{S}) = \int_M\int_0^{\infty}\int^B \phi \exp(-(R^{S}+t\nabla^{S}\phi+t^2))\,dt.\]
Let us call $V$ the pointwise model of the tangent space of $M$, and $E \coloneqq [\Delta]$.
Since both spaces are Euclidean, we can identify them with the respective dual spaces.
Set $e_0 \coloneqq 1$, and $\phi_i \coloneqq e_i\phi$ for $i=0,\dots,7$.
Let $\phi_{ij} \coloneqq \phi_i \wedge \phi_j$. Recall that the term $\nabla \phi$ was computed in \eqref{eq:intrinsic-end1}--\eqref{eq:intrinsic-end2} as
\begin{align*}
\nabla^S \phi & = -\tfrac{a}{4}e_2 \hatotimes \phi_7 + \tfrac{a}{4}e_3 \hatotimes \phi_6 - \tfrac{a}{4}e_6 \hatotimes \phi_3 + \tfrac{a}{4}e_7 \hatotimes \phi_2,
\end{align*}
and the curvature form can be derived from it as
\begin{align*}
R^S & = e_{12} \hatotimes f_{12} + e_{13} \hatotimes f_{13} + e_{16} \hatotimes f_{16} + e_{17} \hatotimes f_{17} + e_{23} \hatotimes f_{23}\\
& \qquad + e_{26} \hatotimes f_{26} + e_{27} \hatotimes f_{27} + e_{36} \hatotimes f_{36} + e_{37} \hatotimes f_{37} + e_{67} \hatotimes f_{67},
\end{align*}
where the $f_{ij}$'s are certain non-vanishing polynomial combinations of $\phi_{\alpha \beta}$. 
Observe that $\nabla^S \phi$ and $R^S$ are of the form $\sum P_{i\alpha}e_i \hatotimes \phi_{\alpha}$ and $R^S = \sum R_{ij\alpha \beta} e_{ij} \hatotimes \phi_{\alpha \beta}$ respectively, hence 
\begin{align*}
[R^S,\nabla^S \phi] & = \sum P_{k\gamma}R_{ij\alpha \beta} [e_{ij}\hatotimes \phi_{\alpha \beta}, e_k \hatotimes \phi_{\gamma}] \\
& = \sum P_{k\gamma}R_{ij\alpha \beta} (e_{ijk} \hatotimes \phi_{\alpha \beta \gamma} - e_{kij} \hatotimes \phi_{\gamma \alpha \beta}) = 0,
\end{align*}
and clearly $R^S$ and $\nabla^S \phi$ commute with $1$. This allows us to write the integral as
\[\int_0^{\infty} e^{-t^2}\left(\int^B \phi\exp(-R^S)\exp(t\nabla^S\phi)\right)\,dt.\]
The product of the exponentials in the Berezin integral is 
\[\left(1-R^S+\tfrac{1}{2!}(R^S)^2-\tfrac{1}{3!}(R^S)^3+\dots \right)\left(1-t\nabla^S\phi+\tfrac{t^2}{2!}(\nabla^S\phi)^2-\tfrac{t^3}{3!}(\nabla^S\phi)^3+\dots\right)\]
and we need to single out the terms of rank $7$. The only terms contributing are then those with the following coefficients:
\begin{align*}
(\nabla^S\phi)^7, \quad R^S\cdot (\nabla^S\phi)^5, \quad (R^S)^2 \cdot (\nabla^S\phi)^3, \quad (R^S)^3 \cdot \nabla^S \phi.
\end{align*}
By the above expressions of $\nabla^S\phi$ and $R^S$, each term does not contain the vectors $e_4$ and $e_5$ in the $\Lambda^*V$ part, so each form cannot be of top degree in the $\Lambda^*V$ part.
Therefore all terms vanish.
\end{proof}
\begin{remark}
Since $e_1\phi$, $e_4\phi$, $e_5\phi$ are harmonic spinors by Proposition \ref{prop:harmonic-spinors}, one can compute the associated $\mathrm G_2$-structures, which turn out to be closed.
It makes sense to compute the Mathai--Quillen current for them as well. 
However, it is well-known that the Mathai--Quillen current is invariant under continous transformations of the spinor $\phi$ keeping the metric fixed.
Indeed, an explicit check shows that the Mathai--Quillen current for $e_i\phi$, $i=1,4,5$ vanishes.
\end{remark}

\begin{remark}
The family of $\mathrm G_2$-structures \eqref{eq:ex1closedg2} has two connected components as $a \neq 0$.
Mapping $a \mapsto -a$ keeps the metric \eqref{eq:left-invariant-metric-general-case1} fixed, and gives the $\mathrm G_2$-structure $\varphi_{-a}$.
Clearly $\varphi_a$ and $\varphi_{-a}$ sit in two different connected components.
Working with $\varphi_{-a}$, one sees that the Mathai--Quillen current vanishes again.
In fact, the intrinsic endomorphism annihilates the span of $e_1,e_4,e_5$, and reverses the signs of \eqref{eq:intrinsic-end1}--\eqref{eq:intrinsic-end2}.
Then Proposition \ref{prop:mathai-quillen-current-vanishes} adapted to this case proves the claim.
Since the other terms in the definition of the $\nu$-invariant do not depend on $a$ or on the unit spinor attached to $\varphi_{-a}$, the $\nu$-invariant vanishes in this case as well.
\end{remark}

\section{Nilmanifolds associated to the Lie algebra \texorpdfstring{$\mathfrak h_2$}{h2}}
\label{sec:nilmanifold2}

We look at the second case by setting things up as in Nicolini \cite[Section 3]{nicolini}.
A Lie algebra $\mathfrak g$ isomorphic to $\mathfrak h_2$ has a basis $(E_1,\dots,E_7)$ satisfying the structure equations $[E_1,E_2]=-E_4$, $[E_1,E_3]=-E_5$, $[E_2,E_3]=-E_6$.
Up to a diagonal automorphism, we can map this basis to a different one $e_1,\dots,e_7$ satisfying 
\[[e_1,e_2]=-ae_4, \qquad [e_1,e_3]=-be_5, \qquad [e_2,e_3]=-ce_6,\]
with $a,b,c \in \mathbb R \setminus\{0\}$. The dual elements $e^i$ satisfy 
\begin{equation} \label{streqh2}  de^ j =0, \, j =1,2,3, 7, \qquad de^4=ae^{12}, \qquad de^5 = be^{13}, \qquad de^6=ce^{23}.\end{equation}
Note that $[\mathfrak g,\mathfrak g]$ is three-dimensional, whereas the centre $\mathfrak z(\mathfrak g) \supset [\mathfrak g,\mathfrak g]$ is again four-dimensional.
Define
\begin{equation}
\label{eq:g2-str-case2-abc}
\varphi \coloneqq e^{123}+e^{145}+e^{167}+e^{246}-e^{257}-e^{347}-e^{356}.
\end{equation}
It is easily checked that $d\varphi=0$ if and only if $a=b+c$.
The metric induced by the above closed $\mathrm G_2$-form is the standard one
\begin{equation}
\label{eq:metric-case2}
g=(e^1)^2+\dots+(e^7)^2,
\end{equation}
so our basis is orthonormal. 
Hereafter, we fix the lattice
\begin{equation}
\label{eq:specific-lattice}
\Gamma \coloneqq \left\{\exp(am_4e_4+bm_5e_5+cm_6e_6)\prod_{k=7,3,2,1}\exp(m_ke_k): m_j \in \mathbb Z\right\},
\end{equation}
as a complete classification of lattices in this case does not seem to be known.

We now adapt some of the statements in Section \ref{sec:nilmanifold1} to this case, without going into the details of all arguments.
We only highlight the differences to the previous case.
Let $G$ be the  simply connected Lie group  with Lie algebra  $\mathfrak g$.
\begin{lemma}
The Levi-Civita connection $\nabla$ on $(G, g)$ is given by
\label{lemma:covariant-derivatives-dirac2}
\begin{alignat*}{2}
\nabla_{e_1}e_2 & = -\nabla_{e_2}e_1 = -\tfrac{a}{2}e_4, && \qquad \nabla_{e_2}e_3 = -\nabla_{e_3}e_2 = -\tfrac{c}{2}e_6 \\
\nabla_{e_1}e_3 & = -\nabla_{e_3}e_1 = -\tfrac{b}{2}e_5, && \qquad \nabla_{e_2}e_4 = +\nabla_{e_4}e_2 = -\tfrac{a}{2}e_1, \\
\nabla_{e_1}e_4 & = +\nabla_{e_4}e_1 = +\tfrac{a}{2}e_2, && \qquad \nabla_{e_2}e_6 = +\nabla_{e_6}e_2 = +\tfrac{c}{2}e_3, \\
\nabla_{e_1}e_5 & = +\nabla_{e_5}e_1 = +\tfrac{b}{2}e_3, && \qquad \nabla_{e_3}e_5 = +\nabla_{e_5}e_3 = -\tfrac{b}{2}e_1, \\
& && \qquad \nabla_{e_3}e_6 = +\nabla_{e_6}e_3 = -\tfrac{c}{2}e_2.
\end{alignat*}
The spin covariant derivative $\nabla^S$ on $(G,g)$ is given by 
\begin{alignat*}{2}
\nabla_{e_1}^S & = \partial_{e_1}-\tfrac{1}{4}(ae_2e_4+be_3e_5), && \qquad \nabla_{e_2}^S = \partial_{e_2}-\tfrac14(ce_3e_6-ae_1e_4), \\
\nabla_{e_3}^S & = \partial_{e_3}+\tfrac14(be_1e_5+ce_2e_6), && \qquad \nabla_{e_4}^S = \partial_{e_4}+\tfrac{a}{4}e_1e_2, \\
\nabla_{e_5}^S & = \partial_{e_5}+\tfrac{b}{4}e_1e_3, && \qquad \nabla_{e_6}^S = \partial_{e_6}+\tfrac{c}{4}e_2e_3, \\
\nabla_{e_7}^S & = \partial_{e_7}
\end{alignat*}
The spin Dirac operator acting on spinors defined on $G$ is given by the formula
\[D = \sum_{k=1}^7 e_k\partial_{e_k}-\tfrac14 (ae_1e_2e_4+be_1e_3e_5+ce_2e_3e_6).\]
\end{lemma} 
From now on we assume $a=b+c$, but let us use $a,b,c$ to keep the expressions simpler. 
The statement of Lemma \ref{lemma:irreps} holds just as well in this case, but the way $G$ acts on the irreducible representations is of course different.
In case (1), $\ell$ is of the form $\ell = \alpha_1e^1+\alpha_2e^2+\alpha_3e^3+\alpha_7e^7$, and the corresponding representation on $\mathbb C$ is given by $\chi_{\ell}(\exp(x)) = e^{2\pi i \ell(x)}$, $x \in \mathfrak g$.
In case (2), a basis of the radical $\mathfrak r_{\ell}$ is given by $\{c\alpha_6e_1-b\alpha_5e_2+a\alpha_4e_3,e_4,e_5,e_6,e_7\}$.
Recall that $a=b+c$ in our case, so the vectors 
\[c(\alpha_6e_1+\alpha_4e_3), \qquad b(-\alpha_5e_2+\alpha_4e_3)\]
add up to $c\alpha_6e_1-b\alpha_5e_2+a\alpha_4e_3$, and are linearly dependent if and only if at least two $\alpha_i$'s vanish.
We distinguish two cases: either $\alpha_4,\alpha_5,\alpha_6 \neq 0$, or at least one of them vanishes.
One then has four subcases:
\begin{enumerate}
\item if $\alpha_4=0$, take $\mathfrak p_{\ell} = \mathrm{Span}_{\mathbb R}\{e_1,e_2,e_4,e_5,e_6,e_7\}$, 
\item if $\alpha_5=0$, take $\mathfrak p_{\ell} = \mathrm{Span}_{\mathbb R}\{e_1,e_3,e_4,e_5,e_6,e_7\}$, 
\item if $\alpha_6=0$, take $\mathfrak p_{\ell} = \mathrm{Span}_{\mathbb R}\{e_2,e_3,e_4,e_5,e_6,e_7\}$, 
\item else, take $\mathfrak p_{\ell} = \mathrm{Span}_{\mathbb R}\{\alpha_6e_1+\alpha_4e_3,-\alpha_5e_2+\alpha_4e_3,e_4,e_5,e_6,e_7\}$.
\end{enumerate}
For brevity, we only treat this last most general case, where an adapted basis needs to be chosen. 
Similar arguments as in the proof of Lemma \ref{lemma:irreps} and Lemma~\ref{lemma:decomposition-l2} can be applied to the other cases, and the final results are analogous. 

When $\mathfrak p_{\ell} = \mathrm{Span}_{\mathbb R}\{\alpha_6e_1+\alpha_4e_3,-\alpha_5e_2+\alpha_4e_3,e_4,e_5,e_6,e_7\}$, it is useful to rescale $\alpha_5 \mapsto 2\alpha_5$ and $\alpha_6 \mapsto 2\alpha_6$, i.e.\ we consider $\ell = \alpha_1e^1+\dots+\alpha_4e^4+2\alpha_5e^5+2\alpha_6e^6+\alpha_7e^7$.
Define a basis 
\begin{equation}
\label{eq:new-choice-basis}
w_1 \coloneqq e_3, \qquad w_2 \coloneqq e_1+\frac{\alpha_4}{2\alpha_6}e_3, \qquad w_3 \coloneqq e_2-\frac{\alpha_4}{2\alpha_5}e_3,
\end{equation}
and $w_k\coloneqq e_k$ for $k=4,5,6,7$.
Clearly $\mathrm{Span}_{\mathbb R}\{w_2,\dots,w_7\}=\mathfrak p_{\ell}$.
One computes 
\[[w_1,w_2] = bw_5, \qquad [w_1,w_3] = cw_6, \qquad [w_2,w_3] = -aw_4+\frac{\alpha_4}{2\alpha_5}bw_5+\frac{\alpha_4}{2\alpha_6}cw_6.\]
Note that $\ell([w_2,w_3])=0$. 
Let $f \colon G \to \mathbb C$ be any function satisfying the equivariance law
\[f(ph) = \chi_{P_{\ell}}(p)f(h) = e^{2\pi i \ell(\log p)}f(p), \qquad p \in P_{\ell}, h \in G.\]
Let $x = \sum_k x_kw_k$. Then
\begin{align*}
\exp(x) & = \exp\left(\sum_{k\neq1} x_kw_k\right) \exp\left(\frac12[x_1w_1,x_2w_2+x_3w_3]\right)\exp(x_1w_1) \\
& = \exp\left(\sum_{k=2}^7x_kw_k+\frac{b}{2}x_1x_2w_5+\frac{c}{2}x_1x_3w_6\right)\exp(x_1w_1),
\end{align*}
and the first factor sits in $P_{\ell}$.
The expression of $f(\exp(x))$ then follows, and the action $\rho_{\ell}$ of $\exp(y_1w_1+\dots+y_7w_7)$ on $f(\exp(tw_1))$ can be computed explicitly. 
With the above choices, the analogue of Lemma \ref{lemma:decomposition-l2} is the following.
\begin{lemma}
\label{lemma:decomposition-second-example}
Each $H_{(\alpha_4,\dots,\alpha_7)}$ decomposes as a direct sum of $G$-invariant summands in the following way. Let $\alpha=(\alpha_4,\dots,\alpha_7)$.
\begin{enumerate}
\item If $\alpha_4^2+\alpha_5^2+\alpha_6^2=0$, then $H_{\alpha}$ splits as a direct sum of subspaces isomorphic to $\mathbb C \otimes_{\mathbb C} \Delta \simeq \Delta$.
Here $G$ acts on the first factor via $\chi_{\ell}$, for \[\ell = \alpha_1e^1+\alpha_2e^2+\alpha_3e^3+\alpha_7e^7, \qquad \alpha_i \in \mathbb Z.\]
\item If $\alpha_4^2+\alpha_5^2+\alpha_6^2>0$, then $H_{\alpha}$ splits as a finite direct sum of subspaces isomorphic to $L^2(\mathbb R,\mathbb C) \otimes_{\mathbb C} \Delta$.
Here $G$ acts on the first factor via $\rho_{\ell}$, for
\[\ell = \sum_{k\neq 4,5,6}\alpha_ke^k+\frac{\alpha_4}{a}e^4+\frac{\alpha_5}{b}e^5+\frac{\alpha_6}{c}e^6, \qquad \alpha_i \in \mathbb Z,\]
with suitable (possibly zero) parameters $\alpha_1,\alpha_2,\alpha_3$.
\end{enumerate}
In both cases, the $G$-action commutes with the action of the complex Clifford algebra $\mathrm{Cl}^{\mathbb C}(7) = \mathrm{Cl}(7) \otimes \mathbb C$ on $\Delta$.
\end{lemma}
\begin{remark}
In point (2) of Lemma \ref{lemma:decomposition-second-example}, when $\alpha_4,\alpha_5,\alpha_6 \neq 0$, it is useful to recall that $\rho_{\ell}$ does not depend on $\ell$ itself, rather on its coadjoint orbit, cf.\ Section \ref{subsec:representation-theory-nilpotent-groups}. 
\end{remark}
The analogue of Proposition \ref{prop:harmonic-spinors} then follows.
\begin{proposition}
\label{prop:harmonic-spinors2}
Let $G$ be connected, simply connected the two-step nilpotent Lie group with Lie algebra isomorphic to $\mathfrak h_2$, and $M = \Gamma \backslash G$ be the associated nilmanifold with $\Gamma$ as in \eqref{eq:specific-lattice}.
For the Riemannian metric $g$ on $M$  induced by the  invariant closed $\mathrm G_2$-structure \eqref{eq:g2-str-case2-abc} with $a=b+c$, and the trivial spin structure, the space of harmonic spinors for the spin Dirac operator $D_M$ is even-dimensional.
Moreover, there exist non-zero, non-invariant harmonic spinors.
\end{proposition}
\begin{proof}
For non-invariant spinors, we only set up the action of the Dirac operator in the most general case above, which exhibits some significant differences to the corresponding one in Lemma \ref{lemma:irreps} and Proposition \ref{prop:harmonic-spinors}. 
However, the computations are completely analogous to those we have already done.

Let $\phi$ be the (unique up to a sign) spinor attached to $\varphi$.
By Lemma \ref{lemma:covariant-derivatives-dirac2}, it is trivial to check that $D_M\phi=0$, $D_M(e_7\phi)=0$, and \begin{alignat*}{2}
D_M(e_1\phi) & = -\tfrac{c}{2} e_6\phi, \qquad && D_M(e_4\phi) = -\tfrac{a}{2}e_3\phi, \\
D_M(e_2\phi) & = +\tfrac{b}{2}e_5\phi, \qquad && D_M(e_5\phi) = +\tfrac{b}{2}e_2\phi, \\
D_M(e_3\phi) & = -\tfrac{a}{2}e_4\phi, \qquad && D_M(e_6\phi) = -\tfrac{c}{2}e_1\phi.
\end{alignat*}
Since $a,b,c$ are non-zero, the space of left-invariant harmonic spinors is generated by $\phi$ and $e_7\phi$.

In case $(1)$ in Lemma \ref{lemma:decomposition-second-example}, the action of $D_M$ on the corresponding irreducible representations $\mathbb C \otimes_{\mathbb C} \Delta \simeq \Delta$ can be discussed with the same method as in Proposition~\ref{prop:harmonic-spinors}.
It turns out the kernel of the Dirac operator in this case is even-dimensional.
Existence of non-trivial, non-invariant harmonic spinors is seen, for instance, when $\alpha_2=\alpha_3=\alpha_7=0$ and $16\pi^2\alpha_1^2=ab=(b+c)b$.

As for the general case $\alpha_4,\alpha_5,\alpha_6 \neq 0$ in Lemma \ref{lemma:decomposition-second-example}, choose a basis $w_1,\dots,w_7$ as in \eqref{eq:new-choice-basis}.
As above, by the Baker--Campbell--Hausdorff formula we have 
\begin{align*}
\exp\left(\sum_{k=1}^7 x_kw_k\right) & = \exp\left(\sum_{k=2}^7x_kw_k+\frac{b}{2}x_1x_2w_5+\frac{c}{2}x_1x_3w_6\right)\exp(x_1w_1)
\end{align*}
and the product of the first two factors sits in $\mathfrak p_{\ell}$.
If \[f(pg) = e^{2\pi i \ell(\log p)}f(g), \qquad p \in \exp(\mathfrak p_{\ell}), g \in G,\]
and $x = \sum_{k=1}^7 x_kw_k$, then the expression of $f(\exp(x))$ can be written explicitly, as well as the action of $\exp(y_1w_1+\dots+y_7w_7)$ on $f(\exp(tw_1))$.
Let $\rho_{\ell}$ be the corresponding action, and observe that
\[e_1 = w_2-\frac{\alpha_4}{2\alpha_6}w_1, \qquad e_2=w_3+\frac{\alpha_4}{2\alpha_5}w_1, \qquad e_3=w_1.\]
Set again $\beta_k = 2\pi i \alpha_k$. 
The differential of $\rho_{\ell}$ can then be computed explicitly, and the Dirac operator then takes the form
\begin{align*}
D_M & = \left(\beta_2+\beta_5t-\frac{\beta_4}{2\beta_6}\frac{d}{dt}\right)e_1+\left(\beta_3+\beta_6t+\frac{\beta_4}{2\beta_5}\frac{d}{dt}\right)e_2+\frac{d}{dt}e_3 \\
& \qquad +\frac{\beta_4}{a}e_4+\frac{\beta_5}{b}e_5+\frac{\beta_6}{c}e_6+\beta_7e_7-\frac14(ae_1e_2e_4+be_1e_3e_5+ce_2e_3e_6).
\end{align*}
Now a similar computation as in the proof of Proposition \ref{prop:harmonic-spinors} gives the claim.
\end{proof}

Some final considerations on the eta-invariants and the Mathai--Quillen current.
Proposition \ref{prop:odd-signature} in Section \ref{sec:eta-invariants} can be easily adapted to this case by taking $\tilde{T}$ as $\tilde{T}e_7 = -e_7$, and $\tilde Te_i = e_i$ for all $i\neq 7$.
Since any lattice $\Gamma$ in $G$ is a product lattice $\Gamma' \times \mathbb Ze_7$, this guarantees that the same argument as in the proof Proposition~\ref{prop:odd-signature} works, so all $\eta$-invariants vanish.
Finally, by the results in Lemma \ref{lemma:covariant-derivatives-dirac2} one computes the intrinsic endomorphism $S$:
\begin{alignat*}{3}
S(e_1) & = +\tfrac{c}{4}e_6, && \qquad S(e_2) = -\tfrac{b}{4}e_5, && \qquad S(e_3) = +\tfrac{a}{4}e_4, \\
S(e_4) & = -\tfrac{a}{4}e_3, && \qquad S(e_5) = +\tfrac{b}{4}e_2, && \qquad S(e_6) = -\tfrac{c}{4}e_1,
\end{alignat*}
and $S(e_7)=0$. Hence we can write
\begin{align*}
\nabla^S\phi & = +\tfrac{c}{4}e_1 \hatotimes \phi_6-\tfrac{b}{4}e_2\hatotimes \phi_5+\tfrac{a}{4}e_3\hatotimes \phi_4 \\
& \qquad -\tfrac{a}{4}e_4 \hatotimes \phi_3+\tfrac{b}{4}e_5\hatotimes \phi_2-\tfrac{a}{4}e_6\hatotimes \phi_1.
\end{align*}
The computation of the Mathai--Quillen current for $\phi$ proceeds in the exact same way as in Section \ref{sec:mathai-quillen-current}.
The result is again zero as $\nabla^S\phi$ takes values in $\mathbb R^6 \hatotimes SM$ and the curvature form $R^S$ takes values in $\Lambda^2 \mathbb R^6 \hatotimes \Lambda^2SM$, where $\mathbb R^6 = \mathrm{Span}\{e_1,\dots,e_6\}$.
So all forms $(\nabla^S\phi)^7$, $R^S \cdot (\nabla^S\phi)^5$, $(R^S)^2\cdot (\nabla^S\phi)^3$, and $(R^S)^3 \cdot \nabla^S\phi$ vanish, as the products in the first factors never reach degree $7$.
It then follows that the same result holds for $e_7\phi$.

\end{document}